\documentclass[12pt,reqno]{amsart}
\usepackage{amssymb}
\usepackage{enumerate}
\usepackage{graphicx}
\usepackage{slashbox}
\usepackage[colorlinks=true, linkcolor=blue, citecolor=magenta]{hyperref}

\usepackage{bm}

\usepackage[T1]{fontenc}
\usepackage[all]{xy}
\usepackage[usenames]{color}
\usepackage[dvipsnames]{xcolor}
\usepackage{algorithm}
\usepackage[noend]{algpseudocode}
\usepackage{tikz}
\usetikzlibrary{shapes,graphs,arrows.meta}

\definecolor{darkgreen}{RGB}{65,165,118}
\definecolor{darkred}{RGB}{180,0,0}

\newcommand{\hadgesh}[1]{\emph{\color{darkred}#1}}

\newcommand{\higherlim}[2]{\displaystyle\setbox1=\hbox{\rm lim}
	\setbox2=\hbox to \wd1{\leftarrowfill} \ht2=0pt \dp2=-1pt
	\setbox3=\hbox{$\scriptstyle{#1}$}
	\def\test{#1}\ifx\test\empty
	\mathop{\mathop{\vtop{\baselineskip=5pt\box1\box2}}}\nolimits^{#2}
	\else
	\ifdim\wd1<\wd3
	\mathop{\hphantom{^{#2}}\vtop{\baselineskip=5pt\box1\box2}^{#2}}_{#1}
	\else
	\mathop{\mathop{\vtop{\baselineskip=5pt\box1\box2}}_{#1}}%
	\nolimits^{#2}
	\fi\fi}

\let\oldcirc=\circ
\renewcommand{\circ}{\mathchoice
    {\mathbin{\scriptstyle\oldcirc}}{\mathbin{\scriptstyle\oldcirc}}
    {\mathbin{\scriptscriptstyle\oldcirc}}
    {\mathbin{\scriptscriptstyle\oldcirc}}}

\SelectTips{cm}{10} \UseTips   %% use computer modern tips in xy

\numberwithin{equation}{section}

\mathcode`\:="003A
\mathchardef\cdot="0201

\renewenvironment{enumerate}[1][]
{\begin{enumerat}[#1]\setlength{\itemsep}{6pt}}{\end{enumerat}}

\renewenvironment{itemize}
{\begin{itemiz}\setlength{\itemsep}{6pt}\setlength{\itemindent}{-20pt}}
{\end{itemiz}}

\def\beq#1\eeq{\begin{equation*}#1\end{equation*}}
\def\beqq#1\eeqq{\begin{equation}#1\end{equation}}

\let\emptyset=\varnothing

\DeclareMathAlphabet\EuR{U}{eur}{m}{n}
\SetMathAlphabet\EuR{bold}{U}{eur}{b}{n}

\renewcommand{\:}{\colon}

\newlength{\upto}\newlength{\dnto}
\newcounter{let} 
\loop\stepcounter{let}
\expandafter\edef\csname cal\alph{let}\endcsname%
{\noexpand\mathcal{\Alph{let}}}
\ifnum\thelet<26\repeat

\newcommand{\tdef}[2][]{\expandafter\newcommand\csname#2\endcsname%
{#1\textup{#2}}}
\tdef{Iso}   \tdef{Aut}    \tdef{Out}    \tdef{Inn}    \tdef{Hom}
\tdef{End}   \tdef{Inj}    \tdef{map}    \tdef{Ker}    \tdef{Ob}
\tdef{Mor}   \tdef{Res}    \tdef{Spin}   \tdef{Id}     \tdef{Fr}
\tdef{rk}    \tdef{conj}   \tdef{incl}   \tdef{proj}   \tdef{diag} 
\tdef{trf}   \tdef{Sol}    \tdef{He}     \tdef{Sz}     \tdef{cj}
\tdef{free}  \tdef{ev}     \tdef{Map}    \tdef{Rep}    \tdef{pr}
\tdef{supp}  \tdef{res}    \tdef{hofibre}  \tdef{Vect}  \tdef{hofib}
\tdef{Coker} \tdef{Sym}
\tdef[_]{fc}  \tdef[_]{fn}  \tdef[_]{typ} \tdef[^]{op}   \tdef[^]{ab}

\newcommand{\fdef}[1]{\expandafter\newcommand\csname#1\endcsname%
{\mathfrak{#1}}}
\fdef{X}  \fdef{red}  \fdef{foc}  \fdef{hyp}  \fdef{Sub}

\newcommand{\bbdef}[1]{\expandafter\newcommand% 
\csname#1\endcsname{\mathbb{#1}}}
\bbdef{C} \bbdef{F} \bbdef{R} \bbdef{Z} \bbdef{Q}  \bbdef{N}

\newcommand{\defeq}{\overset{\textup{def}}{=}}

\newtheorem{Thm}{Theorem}[section]
\newtheorem{Prop}[Thm]{Proposition}
\newtheorem{Cor}[Thm]{Corollary}
\newtheorem{Lem}[Thm]{Lemma}

\newtheorem{Defi}[Thm]{Definition}

\newcommand{\longleft}[1]{\;{\leftarrow%
\count255=0 \loop \mathrel{\mkern-6mu}%
    \relbar\advance\count255 by1\ifnum\count255<#1\repeat}\;}
\newcommand{\longright}[1]{\;{\count255=0 \loop \relbar\mathrel{\mkern-6mu}%
    \advance\count255 by1\ifnum\count255<#1\repeat\rightarrow}\;}

\newcommand{\RIGHT}[3]{\mathrel{\mathop{\kern0pt\longright{#1}}
    \limits^{#2}_{#3}}}

\newcommand{\LEFT}[3]{\mathrel{\mathop{\kern0pt\longleft{#1}}\limits^{#2}_{#3}}
}
\newcommand{\longleftright}[1]{\;{\leftarrow\mathrel{\mkern-6mu}%
    \count255=0\loop\relbar\mathrel{\mkern-6mu}%
    \advance\count255 by1\ifnum\count255<#1\repeat\rightarrow}\;}

\newcommand{\onto}[1]{\;{\count255=0 \loop \relbar\joinrel
    \advance\count255 by1
    \ifnum\count255<#1 \repeat \twoheadrightarrow}\;}

\newcommand{\RLEFT}[3]{\mathrel{%
   \mathop{\vcenter{\baselineskip=0pt\hbox{$\kern0pt\longright{#1}$}%
   \hbox{$\kern0pt\longleft{#1}$}}}\limits^{#2}_{#3}}}

\newtheorem{Rmk}[Thm]{Remark}
\newtheorem{Ex}[Thm]{Example}

\usepackage[usenames]{color}

\usepackage[usenames]{color}                                      %
\long\def\modif#1{{\color{black}#1}} 
 
\long\def\modifone#1{{\color{black}#1}}                           %
\long\def\modifdejan#1{{\color{black}#1}}                           %
\newcommand{\xxto}[1]{\mathrel{\mathop{%
  \setbox0\hbox{$\ {\scriptstyle#1}\ $}%
  \hbox to \wd0{\rightarrowfill}}^{#1}}%
}

\newcommand{\xlto}[2][]{%
  \mathrel{\mathop{%
    \setbox0\vbox{%%\mathsurround=0pt
      \hbox{$\scriptstyle\;\;{#1}\;\;$}%
      \hbox{$\scriptstyle\;\;{#2}\;\;$}%
    }%
    \hbox to\wd0{\leftarrowfill}\displaystyle}%
  \limits^{#2}\ifx{#1}{}\else{_{#1}}\fi}%
}

\newlength{\wdth}

\newcommand{\indeg}{\operatorname{\mathrm{indeg}}}
\newcommand{\outdeg}{\operatorname{\mathrm{outdeg}}}
\newcommand{\sdeg}{\operatorname{\mathrm{sd}}}

\newcommand{\dr}{\operatorname{\mathrm{Dr}}}

\newcommand{\vin}{V_{\mathrm{in}}}
\newcommand{\vout}{V_{\mathrm{out}}}
\newcommand{\sgin}{\sigma_{\mathrm{in}}}
\newcommand{\sgout}{\sigma_{\mathrm{out}}}
\newcommand{\sgio}{\sigma_{\mathrm{in-out}}}

\author{Dejan Govc}
\address{Faculty of Mathematics and Physics, University of Ljubljana,  Ljubljana, Slovenia}
\email{dejan.govc@fmf.uni-lj.si}

\author{Ran Levi}
\address{Institute of Mathematics, University of Aberdeen, Aberdeen, UK}
\email{r.levi@abdn.ac.uk}

\author{Jason P. Smith}
\address{Department of Mathematics, Nottingham Trent University,  Nottingham, UK}
\email{jason.smith@ntu.ac.uk}

\title[Complexes of Tournaments]{Complexes of Tournaments, Directionality Filtrations and Persistent Homology}
\begin{document}
\begin{abstract}
Complete digraphs are referred to in the combinatorics literature as tournaments. We consider a family of semi-simplicial complexes, that we refer to as ``tournaplexes'', whose simplices are tournaments. In particular, given a digraph $\calg$, we associate with it a ``flag tournaplex'' which is a tournaplex containing the directed flag complex of $\calg$, but also the geometric realisation of cliques that are not directed. We define several types of filtrations on tournaplexes, and exploiting persistent homology, we observe that flag tournaplexes provide finer means of distinguishing graph dynamics than the directed flag complex. We then demonstrate the power of these ideas by applying them to graph data arising from the Blue Brain Project's digital reconstruction of a rat's neocortex.
\end{abstract}

\maketitle
\setcounter{figure}{0}
The idea of this article arose from considering topological objects associated to directed graphs, or \hadgesh{digraphs}. Digraphs have been playing an ever increasing role in the study of networks where connections between nodes have a prescribed direction. A notable example of such networks are neural networks, where the use of digraphs is particularly prevalent \cite{Rubinov-Sporns}. The digraphs we consider in this article are finite and simple  i.e., loop-free and any pair of distinct vertices may be connected by reciprocal edges, but not by more than one edge in the same direction.

An \hadgesh{$n$-clique} is a complete graph on $n$-vertices. Any undirected graph $\calg$ gives rise to a simplicial complex called the clique complex, or flag complex, where $n$-simplices are the $(n+1)$-cliques in $\calg$. The clique complex is a very common construction that can be found in  both theoretical and applied works and is highly prevalent in topological data analysis. However, in many contexts one needs to deal with digraphs rather than undirected graphs, in which case forgetting edge orientation will imply loss of information. One solution for this problem was proposed in \cite{Fund} with the introduction of the directed flag complex.

By analogy to the undirected case, a \hadgesh{directed $n$-clique} is a digraph, whose underlying undirected graph is an $n$-clique, and such that the orientation of its edges defines a linear order on its vertices. Given a digraph $\calg$, one associates with it the \hadgesh{directed flag complex} \cite{Fund}, that is, the complex whose $n$-simplices are the directed $(n+1)$-cliques in $\calg$. The directed flag complex was used successfully in \cite{Fund} as a way of gaining information about structural and functional properties of the Blue Brain Project reconstruction of the neocortical microcircuitry of a juvenile rat~\cite{Mega}. However,  while more naturally fitting for use with digraphs, this construction still raises two natural problems. 
The first is that, while edge orientation dictates which cliques are used in the construction of the complex, once the construction is done, it plays no further role in computations.
A second problem is that the directed flag complex does not give a proper topological representation to cliques that are not directed, i.e. those cliques whose vertices are not linearly ordered via the orientation of their edges. This leads to the central idea of this article, namely to the construction of a  complex built of all possible cliques in the digraph, while preserving partial information that is determined by edge orientation.

Finite digraphs whose underlying undirected graphs are cliques are referred to in the combinatorics literature as \hadgesh{tournaments}.  Tournaments have been studied since the 1940's by many authors from an applicable point of view \cite{Alway, KBS, Moran, Landau}, as well as theoretically \cite{Moon}. Tournaments naturally lend themselves to be considered as standard simplices, since an induced subgraph of a tournament is clearly also a tournament. This lead us to the idea that there is a natural family of complexes that are built out of tournaments. We refer to such complexes as \hadgesh{tournaplexes}. The directed flag complex of a digraph, and in fact any ordered simplicial complex is in particular a tournaplex. By analogy to the flag complex of a graph and the directed flag complex of a digraph, we introduce the \hadgesh{flag tournaplex} associated to a digraph. The flag tournaplex of a digraph $\calg$ contains the directed flag complex of $\calg$ as a subcomplex, as well as other naturally occurring sub complexes. 

A key advantage of the flag tournaplex, over other combinatorial constructions one can associate to a digraph, is that it invites the use of persistent homology \cite{EH}, where the filtration parameters arise naturally from the digraph in question, without any extra weights. Furthermore, this can be done in a number of different ways, which makes the flag tournaplex, at least in theory, a more powerful instrument by which one can study properties of digraphs and networks. The idea arose from considering a simple numerical invariant of tournaments that was introduced in \cite{Fund} - the \hadgesh{directionality} of a tournament.

In this paper we define three types of directionality, two that are local in nature and do not depend on an ambient tournaplex the tournament is a part of, and a third that does depend on the containing tournaplex. Local directionality invariants divide tournaments into equivalence classes, and since they are always defined by means of the orientation of the edges in a tournament, they could potentially provide an efficient dimension reduction method in studying  dynamics on a digraph. Since the number of non-isomorphic $n$-tournaments grows very rapidly with $n$, such a method is essential for any applicable purposes. However, in this work we concentrate on the filtrations of tournaplexes that arise from directionality invariants, examine the way in which certain simple examples behave with respect to those filtrations, and apply the technique to some more involved data sets.  
Thus, we aim to convince the reader that the flag tournaplex of a digraph, combined with its directionality invariants and persistent homology, are a powerful tool in the study of neuroscience, network theory, and in general any subject that involves digraphs.
Furthermore,  we show that these methods have the potential to reveal more information about a digraph than some other, more standard, topological tools.

After introducing the concept of a tournaplex in Section \ref{Sec-tournaments}, we proceed in Section \ref{Sec-directionality} by defining  several types of directionality. We mainly discuss three  types: local directionality, 3-cycle directionality and global directionality. We concentrate on the first two, which are closely related to each other, as demonstrated in Proposition \ref{Prop-directionality=D-8k}. We then discuss some basic properties of local and 3-cycle directionality, and in particular derive in Proposition \ref{Prop-expectation} a probabilistic formula for the expectation of tournaments of a given order (number of vertices) and local directionality in a random digraph. 

In Section \ref{Sec-filtrations} we show how directionality can be used to define  filtrations on tournaplexes. In particular we produce simple examples that show the potential these filtrations have to distinguish tournaplexes with the use of persistent homology in cases where the homotopy types of the corresponding geometric realisations are identical. We note that each isomorphism type of a tournament gives rise to a directionality invariant on tournaplexes. This is similar to the idea behind 3-cycle directionality, which is basically an invariant that associates with a tournament the number of regular 3-tournaments it contains (See Remark \ref{Rmk-different-filt}).
In particular we observe (Figures \ref{fig-sphere-abs-dir} and  \ref{fig-sphere-rel-dir}) that different directionality filtrations can be used in conjunction to achieve a greater separation power among tournaplexes. Furthermore, the variety of directionality filtrations on tournaplexes suggests the possibility of associating multi-parameter persistence modules to them. This of course raises the question whether a multi-parameter persistence module may contain more information about the tournaplex than each of the filtrations used to create it individually. We demonstrate by means of a simple example two 8-tournaments whose local directionality filtration and 3-cycle filtration each give isomorphic 1-parameter persistence modules, while taken together they are distinguished by their corresponding 2-parameter persistence modules. In these examples we do not compute the full persistence module structure (i.e., we don't compute the morphisms in the module), as the values at the bifiltration points suffice to separate the samples.

Finally in Section \ref{Sec-applications} we give two examples of applications, where we use the flag tournaplex of a digraph and the associated local directionality filtration as a classifier of certain families of graphs. 

The first family is one of random graphs generated by a specific regime that depends on two parameters $p$ and $q$. We show that applying the method to 200 such random graphs with a fixed value of $p$ and four different values of $q$, we obtain an almost perfect separation of the graphs into four families using a technique as simple as $k$-means clustering. This is used as a test case,  which shows that the technique we develop here is indeed sensitive to network structure. 

A second example is taken from data that was generated by the Blue Brain Project for the paper \cite{Fund}. Here we have response signals  of a digital reconstruction of neocortical tissue of a rat \cite{Mega} to a collection of stimuli that depend on two parameters and appear in 5 seeds (repeats of the same experiment). This gives 45 spike train datasets, which we then use to generate a family of 450 tournaplexes. Once more, $k$-means clustering is applied to separate the signals almost perfectly with respect to one of the two parameters and across all seeds.  In both cases we compare the performance of this technique to that of a similar approach using only the Betti numbers of the corresponding directed flag complex (as was done, using a much larger spiking dataset and without an attempt at classification,  in \cite{Fund}), and observe that using tournaplexes with local directionality filtration gives much better results. See Section \ref{Sec4.2} for more detail about the Blue Brain Project model, and the data used in this example.

We emphasise that the methods proposed in this article are applicable to any system or phenomenon that can be encoded on a digraph. In many such applications researchers consider \hadgesh{motifs}, namely subgraphs on a small number of vertices, as a means of characterising various graph regimes and graph dynamics. Tournaments are nothing but motifs whose underlying undirected graph is a clique, and can thus be considered as building blocks of any other motif. Directionality invariants give ways of controlling the distribution of tournaments in a digraph, without getting lost in the abundance of their isomorphism types, and the associated filtrations provide the means of studying persistent homology on a digraph with intrinsically defined weights. These features are independent of any specific application one may wish to apply the methods to, rather than only to neuroscience.

Most of our computations were carried out by purpose designed modification of the software package \textsc{Flagser} \cite{Flagser}, designed for computation of persistent homology of directed flag complexes. The package which we named \textsc{Tournser} is designed specifically for the computation and manipulation of flag tournaplexes associated to digraphs. \textsc{Flagser}  and \textsc{Tournser} are freely available at  \cite{Flagser-code} and \cite{Tournser-code}, respectively, and interactive online versions are available at~\cite{Flagser-live} and \cite{Tournser-live}. 

As a ``spin-off'' project, the first named author investigated the homotopy types that may arise as directed flag complexes associated to tournaments. He  showed in particular that this gives an invariant of tournaments that provides a finer partition within a given directionality class of tournaments \cite{Govc}. The methods in this paper were instrumental in finding the example in Section~\ref{Sec-filtrations} demonstrating that a multi-parameter persistence module of tournaplexes associated to two or more directionality invariants can be more informative than the the corresponding 1-parameter modules associated to each of them separately. 

The authors acknowledge support from EPSRC, grant  EP/P025072/ - ``Topological Analysis of Neural Systems'', and from \'Ecole Polytechnique F\'ed\'erale de Lausanne via a collaboration agreement with the second named author. \modif{We also thank J\=anis Lazovskis, Henri Riihim\"aki and Pedro Concei\c{c}\~ao for useful discussions. }
 
%%%%%%%%%%%%%%%%%%%%%%%%%%%%%%%%%%%%%%%%%%%%%
%%%%% SECTION ONE %%%%%%%%%%%%%%%%%%%%%%%%%%%%%%%%
%%%%%%%%%%%%%%%%%%%%%%%%%%%%%%%%%%%%%%%%%%%%%
%%%%%%%%%%%%%%%%%%%%%%%%%%%%%%%%%%%%%%%%%%%%%
%%%%%%%%%%%%%%%%%%%%%%%%%%%%%%%%%%%%%%%%%%%%%
%%%%%%%%%%%%%%%%%%%%%%%%%%%%%%%%%%%%%%%%%%%%%
%%%%%%%%%%%%%%%%%%%%%%%%%%%%%%%%%%%%%%%%%%%%%
%%%%%%%%%%%%%%%%%%%%%%%%%%%%%%%%%%%%%%%%%%%%%
%%%%%%%%%%%%%%%%%%%%%%%%%%%%%%%%%%%%%%%%%%%%%
%%%%%%%%%%%%%%%%%%%%%%%%%%%%%%%%%%%%%%%%%%%%%

\section{Tournaments and Tournaplexes}
\label{Sec-tournaments}
We start by defining the basic objects of study and some of their properties. All digraphs considered in this article are assume to be finite and simple. By simple we mean \hadgesh{loop-free}, namely edges of the form $(v,v)$ are not allowed for any vertex $v$, and if $v$ and $w$ are two distinct vertices then a digraph may contain the \hadgesh{reciprocal} edges $(v,w)$ and $(w,v)$, but  two edges in the same orientation between $v$ and $w$  are not allowed.

If $\calg=(V,E)$ and $\calg'=(V',E')$ are digraphs, then a morphism $f\colon\calg\to\calg'$ is a pair of functions $(\phi, \psi)$, where $\phi\colon V\to V'$ and $\psi\colon E\to E'$, such that if $e=(v,v')$ is a directed edge in $\calg$, then $\psi(e) = (\phi(v),\phi(v'))$ is a directed edge in $\calg'$. If $\calg$ is any digraph,  then by  its \hadgesh{underlying undirected graph} we mean  the graph $\widehat{\calg}$ on the same vertices and edges, where edge orientation is ignored \modifdejan{(i.e. any directed edge or pair of reciprocal edges is replaced by a single undirected edge)}.

\begin{Defi}\label{D-tournament}
For a non-negative integer $n$, an  $n$-tournament is a digraph with no reciprocal edges whose underlying undirected graph is an $n$-clique.
\end{Defi}

We will generally denote tournaments by greek letters, $\sigma, \tau$ etc.

\begin{Defi}\label{Def-transitive-regular}
An $n$-tournament $\sigma$ is said to be 
\begin{itemize}
\item \hadgesh{transitive}, if its edge orientation defines  a total ordering on its vertex set,
\item  \hadgesh{semi-regular}, if for each vertex $v\in\sigma$ its in-degree and out-degree differ by at most 1, and 
\item  \hadgesh{regular} if for each vertex $v\in\sigma$ its in-degree and out-degree are equal.
\end{itemize}
\end{Defi}

Clearly a regular tournament must be odd, and a regular tournament is always semi-regular. A tournament containing a vertex whose out-degree and in-degree differ by exactly one must be even.  

Let $\sigma$ be an $n$-tournament. Then for any set of $k$ vertices in $\sigma$, the induced subgraph is a $k$-tournament. Such sub-tournaments will be referred to as the \hadgesh{faces of $\sigma$}.

We can now define complexes built out of tournaments in a way that is totally analogous to the definition of an ordinary abstract simplicial complex. 

\begin{Defi}\label{def-tournaplex}
A \hadgesh{tournaplex} is a  collection $X$ of tournaments, such that if $\sigma\in X$ and $\sigma'$ is a face of $\sigma$, then $\sigma' \in X$. 
\end{Defi} 

If $X$ is a finite tournaplex, then the \hadgesh{dimension of $X$} is the largest integer $n$, such that $X$ contains at least one $(n+1)$-tournament and no $k$-tournaments for $k>n+1$. Thus we will sometimes refer to $(n+1)$-tournaments as \hadgesh{$n$-dimensional}.

\subsection{Geometric realisation}
One may think of tournaplexes abstractly as  in Definition \ref{def-tournaplex}. It is also possible to associate a topological space, a chain complex and homology with a tournaplex. To do so, the simplest way is to think of tournaplexes as semi-simplicial sets.

\begin{Defi}\label{Def:semi-simplicial}
Let $X$ be a tournaplex together with a fixed linear ordering on its set of vertices (1-tournaments). Define a semi-simplicial set $X_\bullet$, whose $n$-simplices are the set $X_n$ of all $(n+1)$-tournaments in $X$.  For every tournament $\sigma\in X$, the vertices of $\sigma$ inherit a total order $v_0, v_1,\ldots, v_n$. For all~$n>0$ and $0\le i\le n$, define face operators $\partial_i\colon X_n\to X_{n-1}$, where  $\partial_i(\sigma)$ is the sub-tournament spanned by  all $v_j$ except $v_i$. 
\end{Defi}

The face operators $\partial_i$ clearly satisfy the usual simplicial face relations, and hence $X_\bullet$ is a well defined semi-simplicial set. Clearly, up to isomorphism of semi-simplicial sets, this does not depend on the choice of total ordering on $X_0$. 

\begin{Defi}\label{Def:geometric-realisation}
The \hadgesh{geometric realisation} of a tournaplex $X$ is defined to be the geometric realisation of the corresponding semi-simplicial set $X_\bullet$. 
\end{Defi}

The \hadgesh{chain complex }of a tournaplex $X$ and its \hadgesh{homology} are simply the usual chain complex and homology of the corresponding semi-simplicial set. 

We proceed with some naturally occurring examples.
\begin{Ex}\label{Ex-simp-cx}
Let $X$ be an abstract simplicial complex, and let $X_1$ denote the collection of its 1-simplices. Fix an arbitrary orientation on each simplex $\tau\in X_1$. Thus every simplex $\sigma\in X$ becomes a tournament, and  turns $X$ into a tournaplex. 
\end{Ex}

Notice that the homeomorphism type of the geometric realisation of the resulting tournaplex in Example \ref{Ex-simp-cx} is independent of the choice of orientations on the 1-simplices. It  is in fact just that of the original simplicial complex.

\begin{Ex}\label{Ex-ordered-simp-cx}
Let $X$ be an ordered simplicial complex, i.e. a collection of finite ordered sets that is closed under subsets. Then each ordered simplex $\sigma\in X$ can be thought of as a transitive tournament. Hence ordered simplicial complexes are special cases of tournaplexes in which every tournament is transitive. 
\end{Ex}

\begin{figure}[b]\label{fig:tournaplex}
	\includegraphics[width=0.7\textwidth]{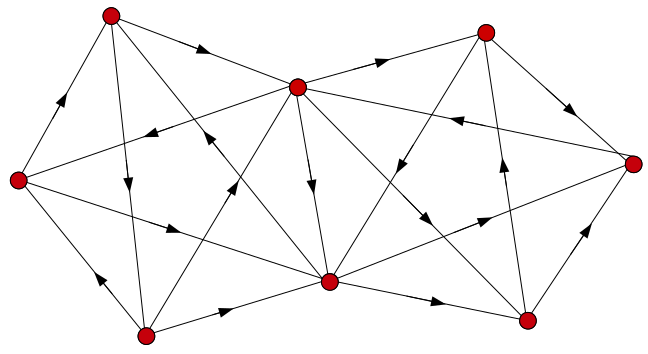}
	\caption{A digraph consisting of two 5-tournaments ``glued'' along a 1-face. The flag tournaplex (Definition \ref{Def-disimpicial-flag-cx}) of this graph consists of two 4-simplices glued along an edge, and as such is contractible. By contrast, the directed flag complex of the graph has the homotopy type of a wedge of two circles and a 2-sphere. However, the topology of the directed flag complex that is embedded in the flag tournaplex can be revealed using the 3-cycle filtration (see Section~\ref{Sec-filtrations}).}
\end{figure}

\newcommand{\tfl}{\mathrm{tFl}}
\newcommand{\dfl}{\mathrm{dFl}}

\begin{Defi}\label{Def-disimpicial-flag-cx}
Let $\calg$ be a digraph. The  \hadgesh{flag tournaplex} associated to $\calg$ is the tournaplex $\tfl(\calg)$, whose $n$-simplices are the subgraphs of $\calg$ that are $(n+1)$-tournaments. 
\end{Defi}

Notice that not every tournaplex is a flag tournaplex. For instance the boundary of a 2-simplex, with any edge orientation is a tournaplex because it consists of three 2-tournaments  and three  1-tournaments, but it is not a flag tournaplex because it does not contain the interior of the 2-simplex. However, just like the case of ordinary flag complexes, every tournaplex is contained as a sub-tournaplex in the flag tournaplex generated by its 1-skeleton. 

The \hadgesh{directed flag complex} associated to a digraph $\calg$ is the ordered simplicial complex~$\dfl(\calg)$ whose simplices are the transitive tournaments in $\calg$. Thus $\dfl(\calg)\subseteq \tfl(\calg)$, where equality holds if and only if every tournament in $\calg$ is transitive. Moreover, if $\calg$ has no reciprocal edges then $\tfl(\calg)$ is isomorphic to the flag complex of the underlying undirected graph.

We end this section with a comment on the number of tournaments up to digraph isomorphism. In low dimensions the numbers are small: the number of non-isomorphic tournaments on 2, 3, 4 and 5 vertices is 1, 2, 4, and 12, respectively. However the numbers grow rapidly with dimension with nearly 200,000 non-isomorphic tournaments on 9 vertices and nearly 10 million of them on 10 vertices. It is clear therefore that it may be useful to find invariants that can divide tournaments into a relatively small number of classes in a meaningful way. This is what we aim to do in Section~\ref{Sec-directionality}.

%%%%%%%%%%%%%%%%%%%%%%%%%%%%%%%%%%%%%%%%%%%%%%%%%%%%%%%%
%%%%%%%% SECTION TWO %%%%%%%%%%%%%%%%%%%%%%%%%%%%%%%%%%%%%%%%
%%%%%%%%%%%%%%%%%%%%%%%%%%%%%%%%%%%%%%%%%%%%%%%%%%%%%%%%
%%%%%%%%%%%%%%%%%%%%%%%%%%%%%%%%%%%%%%%%%%%%%%%%%%%%%%%%
%%%%%%%%%%%%%%%%%%%%%%%%%%%%%%%%%%%%%%%%%%%%%%%%%%%%%%%%
%%%%%%%%%%%%%%%%%%%%%%%%%%%%%%%%%%%%%%%%%%%%%%%%%%%%%%%%
%%%%%%%%%%%%%%%%%%%%%%%%%%%%%%%%%%%%%%%%%%%%%%%%%%%%%%%%
%%%%%%%%%%%%%%%%%%%%%%%%%%%%%%%%%%%%%%%%%%%%%%%%%%%%%%%%
%%%%%%%%%%%%%%%%%%%%%%%%%%%%%%%%%%%%%%%%%%%%%%%%%%%%%%%%
%%%%%%%%%%%%%%%%%%%%%%%%%%%%%%%%%%%%%%%%%%%%%%%%%%%%%%%%

\section{Directionality invariants of tournaments}\label{Sec-directionality}
Given a digraph $\calg =(V,E)$, one may associate with each vertex $v\in V$ its signed directionality. This allows us to define certain integer valued functions on the set of simplices of a tournaplex. This in turn gives natural ways of filtering tournaplexes, and as such they become natural candidates for analysis by means of techniques of persistent homology. 

\begin{Defi}\label{def:directionality-functions}
Let $\calg = (V, E)$ be a digraph. For a vertex $v\in V$, define the \hadgesh{signed degree of $v$ in $\calg$}, by
\[\sdeg_\calg(v) = \indeg_\calg(v) - \outdeg_\calg(v).\] 
Let $U\subseteq V$ be a subset of vertices. 
\begin{enumerate}[i)]
	\item Define the \hadgesh{signed degree of $U$ relative to $\calg$} by 
	\[\sdeg_\calg(U) = \sum_{v\in U} \sdeg_\calg(v).\]
	\label{2.1i}
	\item Define the \hadgesh{directionality of $U$ relative to $\calg$} by
	\[\dr_\calg(U) = \sum_{v\in U} \sdeg_\calg(v)^2.\]
\end{enumerate}
\end{Defi}

Let $U$ be a finite set, and let $\R U$ be the real vector space generated by $U$, and let $\R U^*$ be the dual space of linear functionals $\R U \to \R$.   Equip $\R U^*$ with an inner product defined by  
\[(\alpha, \beta) = \sum_{v\in U} \alpha(v)\beta(v)\]
for $\alpha, \beta \in \R U^*$. If $U$ is the vertex set of a digraph $\calg$, then the functions $\indeg$, $\outdeg$ and $\sdeg$ can be extended to functionals 
$\indeg_\calg, \,\outdeg_\calg, \,\sdeg_\calg \in \R U^*$, and $\dr_\calg(U) = (\sdeg_\calg, \sdeg_\calg)$. In the norm on $\R U^*$ defined by the inner product, $\dr_\calg(U)$ is the square length of the functional $\sdeg_\calg$, or equivalently the square distance between the functionals $\indeg_\calg(U)$ and $\outdeg_\calg(U)$. Clearly, if $f\colon \calg\to\calh$ is an isomorphism of digraphs, and $W=f(U)$, then $\sdeg_\calg(U) = \sdeg_\calh(W)$ and $\dr_\calg(U)= \dr_\calh(W)$. We now use these constructions to define graph invariants on tournaments.

\begin{Defi}
For any $n$-tournament $\sigma$ in a graph $\calg$, let $V_\sigma$ denote the vertex set of $\sigma$, and define:
\begin{enumerate}[i)]
\item \hadgesh{Local directionality}:   $\dr(\sigma) = \dr_{\sigma}(V_\sigma)$.
\label{Directionality}

\item \hadgesh{3-cycle directionality}: Let  $c_3(\sigma)$ denote the number of regular 3-sub-tournaments in $\sigma$. 
\label{3-cycle}

\item \hadgesh{Global directionality:}  $\dr_\calg(\sigma) = \dr_\calg(V_\sigma)$.
\end{enumerate}
\label{Tournament-Inv}
\end{Defi}

\newcommand{\ii}{\mathrm{i}}
\newcommand{\oo}{\mathrm{o}}
\newcommand{\iis}{\mathrm{i}_\sigma}
\newcommand{\oos}{\mathrm{o}_\sigma}
\newcommand{\iisp}{\mathrm{i}_{\sigma'}}
\newcommand{\oosp}{\mathrm{o}_{\sigma'}}
\newcommand{\iit}{\operatorname{\ii_\tau}}
\newcommand{\oot}{\operatorname{\oo_\tau}}
\newcommand{\iitp}{\operatorname{\ii_{\tau'}}}
\newcommand{\ootp}{\operatorname{\oo_{\tau'}}}
\newcommand{\rel}{\mathrm{rel}}
\newcommand{\abs}{\mathrm{abs}}

Notice that local directionality and the 3-cycle directionality are both invariants of a single tournament, independent from the ambient graph containing it. Global directionality on the other hand takes into account the general connectivity of the ambient graph. Next, we observe that the local directionality of a tournament is strongly related to its 3-cycle directionality.

\newcommand{\maxdir}[1]{2{\binom{#1+1}{3}}}

\begin{Prop}\label{Prop-directionality=D-8k} Let  $\sigma$ be an $n$-tournament.  Then 
	\begin{equation}\dr(\sigma) =  \maxdir{n} - 8c_3(\sigma). \label{directionality=dr-8k}\end{equation}
\end{Prop}
\begin{proof}
	For $n=3$ the tournament $\sigma$ is either transitive or regular. In the first case $\dr(\sigma) = 8$ and in the other $\dr(\sigma)=0$. Thus the claim follows.
	Proceed  by induction on $n$. Assume~(\ref{directionality=dr-8k}) holds for all $(n-1)$-tournaments and prove \modif{it holds} for $n$-tournaments. Choose a vertex $v_0\in\sigma$, and let $\sigma'$ denote the face of $\sigma$ corresponding to removing $v_0$. We split the vertices $V_{\sigma'}$ of $\sigma'$ into two parts: $\vin$ consisting of the $v\in V_{\sigma'}$ such that $[v_0\to v]\in\sigma$ and $\vout$ consisting of all vertices  $v\in V_{\sigma'}$ such that $[v\to v_0]\in\sigma$.
	
	Define $\sgin$ to be the subgraph of $\sigma'$ induced by $\vin$, and let $\sgout$ be the subgraph induced by $\vout$. Let $\sgio$ be the subgraph with vertex set $V_{\sigma'}$, and whose edges are incident to one vertex in $\vin$ and the  other in $\vout$. 
	
	Now, write:
	\begin{align*}
	\sum_{v\in V_{\sigma}}\sdeg_{\sigma}(v)^2&=
	\sum_{v\in\vin}( \sdeg_{\sigma'}(v)+1)^2+
	\sum_{v\in\vout}( \sdeg_{\sigma'}(v)-1)^2+(|\vout|-|\vin|)^2\\
	&=\sum_{v\in V_{\sigma'}}( \sdeg_{\sigma'}(v))^2+
	2\sum_{v\in\vin} \sdeg_{\sigma'}(v)-2\sum_{v\in\vout} \sdeg_{\sigma'}(v)+|V_{\sigma'}|+(|\vout|-|\vin|)^2.
	\end{align*}
	Note that $\sum_{v\in V_X}\sdeg_{X}(v)=0$ for any graph $X$ with vertex set $V_X$, since this is equivalent to sum of all in-degrees minus the sum of all out-degrees.
	So we can rewrite the second and third term as follows:
	\begin{align*}
	2\sum_{v\in\vin} \sdeg_{\sigma'}(v)-2\sum_{v\in\vout} \sdeg_{\sigma'}(v)&=2\sum_{v\in V_{\sigma'}} \sdeg_{\sigma'}(v)-4\sum_{v\in\vout} \sdeg_{\sigma'}(v)\\
	&=2\sum_{v\in V_{\sigma'}} \sdeg_{\sigma'}(v)-4\sum_{v\in\vout}\sdeg_{\sgout}(v)-4\sum_{v\in\vout}\sdeg_{\sgio}(v)\\
	&=-4\sum_{v\in\vout}\sdeg_{\sgio}(v)\\
	&=4|\vin||\vout|-8\ell,
	\end{align*}
	where $|\vin||\vout|$ arises as the number of all edges in $\sgio$ and $\ell$ is the number of edges $v_1\to v_2$ with $v_1\in\vin$ and $v_2\in\vout$. Therefore,
	\begin{align*}
	\dr(\sigma)=
	\sum_{v\in V_\sigma}\sdeg_{\sigma}(v)^2&=
	\dr(\sigma')-8\ell+|V_{\sigma'}|+(|\vout|-|\vin|)^2+4|\vin||\vout|\\
	&=\dr(\sigma')-8\ell+|V_{\sigma'}|+|V_{\sigma'}|^2.
	\end{align*}
	Using the inductive hypothesis, this simplifies to
	\[
	\dr(\sigma) = 2{\binom{n}{3}}  + n(n-1) - 8(c_3(\sigma') + \ell) = \maxdir{n}- 8(c_3(\sigma') + \ell).\]
	Finally,  observe that $\ell$ is precisely the number of  regular 3-tournaments in $\sigma$ that are  not already present in $\sigma'$, and the proof is complete.
\end{proof}

Note that Proposition~\ref{Prop-directionality=D-8k} can be derived from the first corollary to \cite[Theorem 4]{Moon}. However, we include the above proof as we feel it gives
a simpler combinatorial understanding of the link between local directionality and the  $3$-cycle directionality.

\begin{Cor}[{\cite[Supp. Meth. 2.1, Proposition 1]{Fund}}]\label{lem:transitive_tournament}
	Let $\sigma$ be an $n$-tournament. Then 
	\[\dr(\sigma) \le \maxdir{n},\]
	with equality obtained if and only if $\sigma$ is transitive.
\end{Cor}

\begin{Cor}\label{cor-directionality}
Let $\sigma$ be an  $n$-tournament, and let $\sigma'$ be a face of codimension 1. Then 
\[c_3(\sigma) - c_3(\sigma')\le \left(\frac{n-1}{2}\right)^2.\]
\end{Cor}
\begin{proof}
Let $\sigma'$ be an $(n-1)$-face of $\sigma$,
and let $v_0\in V_\sigma$ be the vertex not present in $V_{\sigma'}$. Partition $V_{\sigma'}$ into two disjoint subsets $\vin$ and $\vout$, where $\vin$ contains the vertices $v\in V_{\sigma'}$  such that the edge between $v_0$ and $v$ is oriented towards $v$, and $\vout=V_{\sigma'}\setminus\vin$. 
Set $\ell=|\vin|$, so $|\vout|=n-1-\ell$. Then the largest value of $c_3(\sigma)-c_3(\sigma')$, that is, the number of directed $3$-cycles in $\sigma$ that are not in $\sigma'$, is obtained if every edge between a vertex $v'\in\vin$ and a vertex $v''\in\vout$ is oriented from $v'$ to $v''$, and that number is $\ell(n-1-\ell) = -\ell^2 + (n-1)\ell$. This quadratic function of $\ell$ obtains its maximum exactly when $\ell=\frac{n-1}{2}$, and so the maximal number of directed $3$-cycles that can appear in $\sigma$ and are not present in $\sigma'$ is $\left(\frac{n-1}{2}\right)^2$.
\end{proof}

\begin{Ex}\label{ex-dr-not-filtration}
If $\sigma$ is an $n$-tournament and $\tau$ is a $k$-face of $\sigma$, then it is not true in general that $\dr(\tau)\le \dr(\sigma)$. For instance a regular 3-tournament has local directionality 0 but each of its 2-faces has local directionality 2. Similarly a semi-regular 4-tournament has local directionality 4 but has two 3-faces with local directionality 8. 
\end{Ex}

One natural question in view of Proposition \ref{Prop-directionality=D-8k}  is whether every number that can theoretically appear as the local directionality of a tournament, is indeed realisable as such. The answer to this question follows from three theorems by Kendall and Babington-Smith in their 1940 paper \cite[8.(1)-(3)]{KBS}. Using our terminology these results read as follows:
\begin{Thm}[\cite{KBS}]\label{Kendall-B-Smith}
	The following statements hold for an $n$-tournament $\sigma$:
	\begin{enumerate}[\rm i)]
		\item 		$c_3(\sigma)\le \begin{cases}
								\frac{n^3-n}{24} & n\; \text{odd}\\
								\frac{n^3-4n}{24} & n\; \text{even}
							  \end{cases}.$
		Furthermore, these bounds are sharp in the sense that there exists an $n$-tournament with this number of 3-cycles. 
		\label{KBS12}

		\item For any  integer $k$ between 0 and the upper bounds in \ref{KBS12}), there exists at least one  $n$-tournament $\sigma$ such that $c_3(\sigma) = k$.
	\end{enumerate}
\end{Thm}

\begin{Cor}\label{Cor-realisability}
	For each pair of  non-negative integers $n$ and $k$, where $k$ is  at most as large as the bounds in Theorem \ref{Kendall-B-Smith}, there exists at least one $n$-tournament $\sigma$ such that $\dr(\sigma) = 2{\binom{n+1}{3}}-8k$. 
\end{Cor}

\begin{Cor}\label{Cor-min-dir} 
For any $n\geq 0$ there exists an $n$-tournament $\sigma$ of minimal local directionality that is regular if $\dr(\sigma)=0$ ($n$ odd)  and semi-regular if $\dr(\sigma)=n$ ($n$  even).
\end{Cor}

\subsection{Distribution of tournaments  in digraphs by local directionality}
The motivation for introducing tournaplexes is the idea that tournaments are basic building blocks in a geometric object that can be associated to a digraph.  Rather than ignoring edge orientation and considering the ordinary flag complex, or neglecting cliques that are not linearly ordered  and considering the directed flag complex, the flag tournaplex allows us to consider all tournaments as simplices, and as such forms a topological object that is richer in structure, and better informing about properties of the digraph in question. \modif{See Figure \ref{fig:simple} for a simple example of how these three constructions differ.}
However, the vast number of isomorphism types of tournaments in higher dimensions makes using them as individual data units quite impractical. We have thus introduced directionality as one way of taking advantage of general tournaments, without the constraints imposed by the large number of isomorphism types. 

\begin{figure}[ht!]
	\centering
	\includegraphics[width=0.2\textwidth]{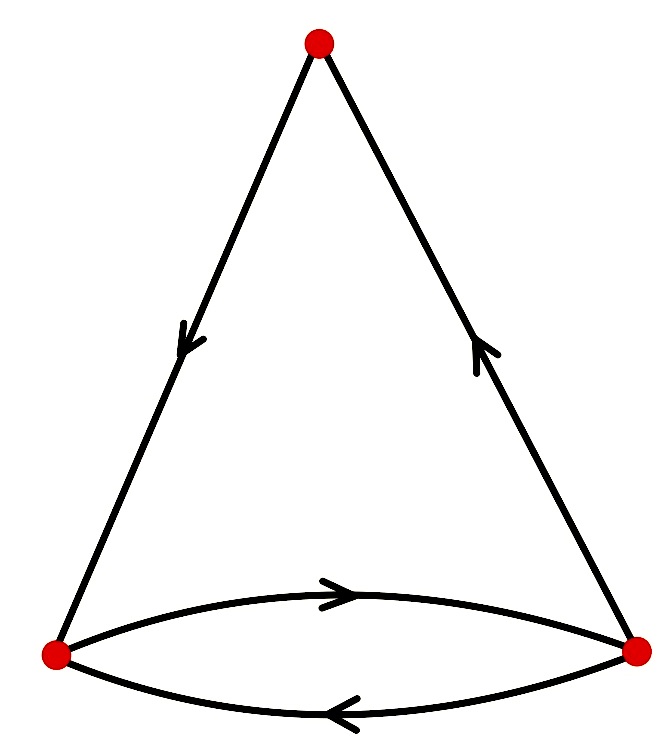}
	\caption{The flag complex of this digraph (ignoring orientation) is a 2-simplex. The directed flag complex is a 2-simplex with an additional arc on its bottom 1-face. The flag tournaplex is a cone formed of two 2-simplices attached on two of their edges. Its homotopy type can be distinguished from that of the flag complex by using the directionality of its 3-tournaments.}
\label{fig:simple}
\end{figure}

To demonstrate the potential of tournaments to inform on network structure, we examined a number of networks, and recorded the distribution of tournaments by local directionality in Figure \ref{fig:directionality-distribution}. In particular the probability of each possible value of directionality of a random  tournament in a fixed dimension is known up to $n=10$, by calculations of Alway \cite{Alway} and Kendall and Babington-Smith \cite{KBS}. More generally it was shown by Moran \cite{Moran} that the distribution of the number of 3-cycles in a random $n$-tournament tends to normal for $n$  sufficiently large (see also \cite[Theorem 6]{Moon}). All these results are stated in terms of the number of 3-cycles in a tournament. 

Let $T_{n,j}$ be the number of all $n$-tournaments $\sigma$ on a fixed vertex set, such that $c_3(\sigma)=j$.
 One can compute the  distribution of local directionality in $n$-tournaments and $T_{n,j}$ recursively \cite{Alway}.
 The following gives the said distributions for $n\le5$.

\begin{center}
\begin{tabular}{||c||c||c||c|c||c|c|c||c|c|c|c|c|c||}
\hline
n & 1 & 2 &3&  &4&&& 5&&&&&\\
\hline
$j$&0&0&0&1&0&1&2&0&1&2&3&4&5\\
\hline
$\dr(\sigma)$&0&2&8&0&20&12&4&40&32&24&16&8&0\\
\hline
$T_{n,j}$&1&2&6&2&24&16&24&120&120&240&240&280&24\\
\hline
\end{tabular}
\end{center}

Using $T_{n,j}$ we can express the expected number of tournaments of various $3$-cycle directionalities in Erd\H{o}s-R\'enyi graphs. \modifdejan{Recall that an} \hadgesh{Erd\H{o}s-R\'enyi digraph} \modifdejan{with $n$ vertices and connection probability $p$ is a random subgraph of the complete digraph with $n$ vertices, where each of the $n(n-1)$ possible directed edges is included with probability $p$ independently from every other directed edge.}

\begin{Prop}
\label{Prop-expectation}
Let $\calg$ be a directed Erd\H{o}s-R\'enyi graph with $n$ vertices and connection probability $p$. Let $X_k$ be the total number of $k$-tournaments that occur in $\calg$ and let $X_{k,j}$ be the total number of $k$-tournaments containing exactly $j$ $3$-cycles that occur in $\calg$. Then the expected values of $X_k$ and $X_{k,j}$ are given by the formulas
\[
E(X_k)=\binom{n}{k}2^{\binom{k}{2}}p^{\binom{k}{2}}\qquad\text{and}\qquad E(X_{k,j})=\binom{n}{k}T_{k,j}p^{\binom{k}{2}}.
\]
\end{Prop}

\begin{proof}
First note that the number of all possible $k$-tournaments on an $n$-vertex set is given by
\[
\binom{n}{k}2^{\binom{k}{2}},
\]
as each one is obtained by choosing $k$ out of $n$ vertices and then orienting each of the $\binom{k}{2}$ edges in one of two possible ways.

The number of all possible $k$-tournaments on an $n$-vertex set containing exactly $j$ $3$-cycles is given by
\[
\binom{n}{k}T_{k,j},
\]
as we have to choose $k$ vertices and then by definition $T_{k,j}$ is the number of tournaments on that vertex set containing exactly $j$ $3$-cycles.

Now, any specific $k$-tournament $\sigma$ will occur in $\calg$ with probability
\[
P(\sigma\subseteq \calg)=p^{\binom{k}{2}},
\]
as it has exactly ${\binom{k}{2}}$ edges, each of which occurs independently with probability $p$.
Let $Y_\sigma$ be the random variable which takes value $1$ if $\sigma$ occurs in $\calg$ and $0$ otherwise. It follows that
\[
E(Y_\sigma)=p^{\binom{k}{2}}.
\]
Note that $X_k$ is just the sum of $Y_\sigma$ over all possible $k$-tournaments $\sigma$. 
Similarly, $X_{k,j}$ is the sum of $Y_\sigma$ over all possible $k$-tournaments $\sigma$ containing exactly $j$ $3$-cycles.
Therefore, the result follows by linearity of expectation.
\end{proof}

We can consider $T_{k,j}$ as the theoretical distribution of $k$-tournaments $\sigma$ with $c_3(\sigma)=j$. In Figure \ref{fig:directionality-distribution} we show the distribution of tournaments in all values of local directionality in a number of sample networks. Notice the rather accurate estimates in the case of Erd\H{o}s-R\'enyi graphs. It is also worth noticing how close the distribution of local directionality values is between the connectivity graph of C. Elegans \cite{celegans} and that of the Blue Brain Project microcircuit simulating a section of the neocortex of a rat \cite{rat}. The bottom three datasets can be found in \cite{Mac, Gplus, Protein}. The bottom five datasets are available in \textsc{Tournser} format at~\cite{Abdn-NT}.

\begin{figure}[ht!]
	\centering
	\includegraphics[width=0.7\textwidth]{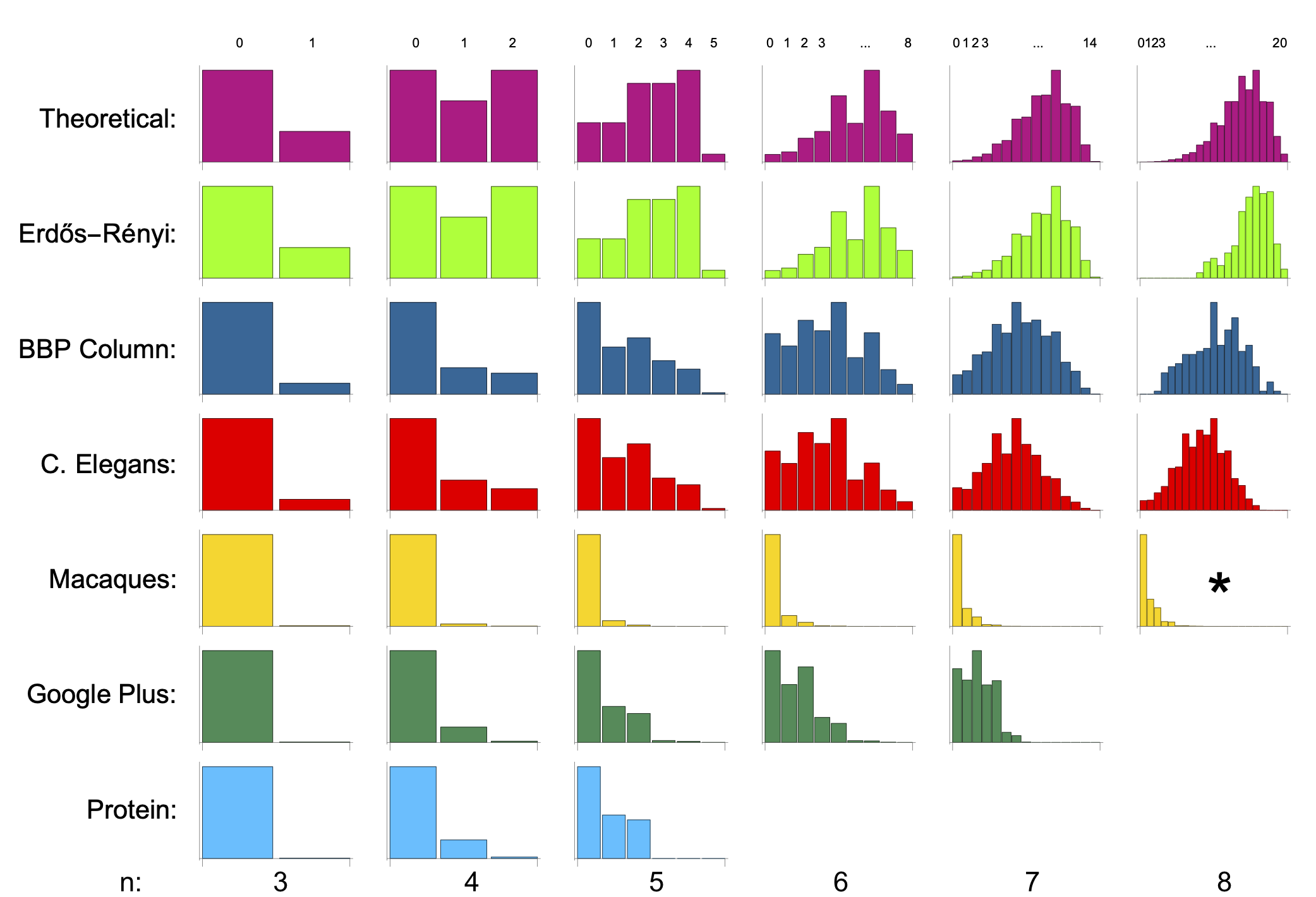}
	\caption{\modif{Distribution of tournaments by local directionality in various networks. The rows correspond to the graphs indicated on the left. The columns are ordered left to right by the number $n$ of vertices in a tournament for $n\geq 3$. Each histogram shows the distribution of $n$-tournaments by local directionality. The numbers on top refer to the number $c_3$  of 3-cycles in the tournaments represented by the relevant  column of the histogram. From this number local directionality can be computed as $2{n+1\choose 3} - 8c_3$. The asterisk indicates  the existence of further distributions to the right. }}
\label{fig:directionality-distribution}
\end{figure}

%%%%%%%%%%%%%%%%%%%%%%%%%%%%%%%%%%%%%%%%%%%%%%%%%%%%%%%%
%%%%% SECTION THREE        %%%%%%%%%%%%%%%%%%%%%%%%%%%%%%%%%%%%%%%%
%%%%%%%%%%%%%%%%%%%%%%%%%%%%%%%%%%%%%%%%%%%%%%%%%%%%%%%%
%%%%%%%%%%%%%%%%%%%%%%%%%%%%%%%%%%%%%%%%%%%%%%%%%%%%%%%%
%%%%%%%%%%%%%%%%%%%%%%%%%%%%%%%%%%%%%%%%%%%%%%%%%%%%%%%%
%%%%%%%%%%%%%%%%%%%%%%%%%%%%%%%%%%%%%%%%%%%%%%%%%%%%%%%%
%%%%%%%%%%%%%%%%%%%%%%%%%%%%%%%%%%%%%%%%%%%%%%%%%%%%%%%%
%%%%%%%%%%%%%%%%%%%%%%%%%%%%%%%%%%%%%%%%%%%%%%%%%%%%%%%%
%%%%%%%%%%%%%%%%%%%%%%%%%%%%%%%%%%%%%%%%%%%%%%%%%%%%%%%%
%%%%%%%%%%%%%%%%%%%%%%%%%%%%%%%%%%%%%%%%%%%%%%%%%%%%%%%%

\section{Filtrations}
\label{Sec-filtrations}

In this section we discuss three  different ways to define a filtration on a tournaplex using the idea of directionality. 

Let $K$ be a tournaplex with vertex set $V$.  An \hadgesh{increasing filtration on $K$} is an increasing sequence of sub-tournaplexes 
\[\emptyset=F_0(K)\subseteq F_1(K)\subseteq F_2(K)\subseteq \cdots \subseteq F_n(K) = K.\]
An \hadgesh{increasing filtering weight function on $K$} is a function  $W\colon K\to\R$ with the property that if $\sigma$ is a tournament in $K$ and $\sigma'\subseteq \sigma$ is a face, then $W(\sigma')\le W(\sigma)$. Similarly one defines decreasing filtrations and decreasing filtering weight functions. 
	
An increasing filtering weight function $W$ on a tournaplex $K$ naturally gives rise to  increasing filtrations on $K$ as follows: Fix an increasing sequence of real numbers $r_1<r_2<\cdots<r_{n-1}$. Set $F_0(K)=\emptyset$, $F_n(K)=K$, and
\[F_i(K) = \{\sigma\in K\;|\; W(\sigma)\le r_i\},\]
for $1\le i\le n-1$.
Similarly a decreasing filtering weight function gives rise to a decreasing filtration on $K$. If the weight function on $K$ is a step function, the one has an obvious choice for  the sequence defining the filtration, namely the sequence of all possible values of the function in increasing order.
We now use local, 3-cycle and global directionality as ways of defining filtrations on any tournaplex. 

We start with local directionality. As  pointed out in Example \ref{ex-dr-not-filtration}, it is not generally the case that if $\sigma$ is a tournament and $\tau$ is a face of $\sigma$, then $\dr(\tau)\le\dr(\sigma)$. Thus one is led to make the following definition. 

\begin{Defi}\label{Def-local-dir-filt}
	Let $K$ be a tournaplex. For each $n$-tournament $\sigma\in K$ define the \hadgesh{local directionality weight of $\sigma$} to be 	
	\[W_{\dr}(\sigma) \defeq \dr(\sigma) + 2{\binom{n}{3}}.\]
\end{Defi}

\begin{Lem}\label{Lem-local-dir-filt-works}
Let $K$ be a tournaplex and let $W_{\dr}\colon K\to \N$ be the local directionality weight function. Then $W_{\dr}$ defines an increasing filtration on $K$.
\end{Lem}
\begin{proof}
It suffices to show that if $\sigma$ is an $n$-tournament for $n>2$ and  $\tau\subseteq\sigma$ is a face of codimension 1, then $W_{\dr}(\tau)\le W_{\dr}(\sigma)$. By Proposition \ref{Prop-directionality=D-8k} and Corollary \ref{cor-directionality}
\begin{align*}\label{Eq-Filtration}
W_{\dr}(\sigma)-W_{\dr}(\tau) \geq  & \;2\left({\binom{n+1}{3}}-{\binom{n-1}{3}}\right) -2(n-1)^2=0.
\end{align*}
The claim follows.
\end{proof}

Notice that it is possible for an $n$-tournament $\sigma$ to have a face $\tau$ of codimension larger than 1, whose directionality is larger than that of all codimension 1 faces. This justifies adding the maximal possible directionality of a codimension 1 face to the local directionality in order to create a filtration. See Figure \ref{fig-Tourn-Ex} for such an example.

\begin{figure}[ht!]
	\centering
	\includegraphics[width=0.3\textwidth]{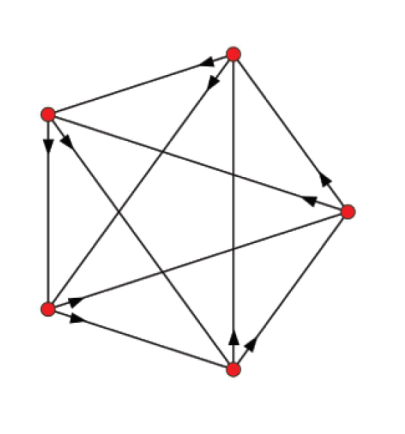}
	\caption{A regular 5-tournament ($\dr=0$) in which each 4-face is semi-regular ($\dr=4$). Hence each 4-face contains two 3-faces that are transitive~($\dr=8$).}
	\label{fig-Tourn-Ex}
\end{figure}

\begin{Defi}\label{Def-3c-dir-filt}
	Let $K$ be a tournaplex. For each $n$-tournament $\sigma\in K$ define the \hadgesh{3-cycle directionality weight of $\sigma$} to be $W_{c_3}(\sigma)=c_3(\sigma)$.
\end{Defi}

Clearly if $\sigma$ is a tournament and $\tau$ is a face of $\sigma$, then $c_3(\tau)\le c_3(\sigma)$, and hence this defines a filtration on any tournaplex.

\begin{Rmk}\label{Rmk-different-filt}
Observe that tournaplexes may be naturally filtered in many other ways that are similar to the 3-cycle filtration. For instance, if $K$ is a tournaplex, define $W_T\colon K\to \R$ by letting $W_T(\sigma)$ be the number of transitive 3-tournaments in $\sigma$.  Pushing this idea further, one can consider any fixed (small) tournament $\tau$, and filter $K$ by the number of times $\tau$ appears  as a face in each $\sigma\in K$. This automatically yields a filtration, and relates nicely to well known approaches that regard the prevalence of certain motifs in networks as a determining factor in the emerging dynamics (see for instance \cite{Schiller-et-al}).
\end{Rmk}

By Proposition \ref{Prop-directionality=D-8k}, the 3-cycle weight function and the local directionality weight function on an arbitrary $n$-tournament $\sigma$  are related by the formula:
\begin{equation}\label{local-vs-c3}
W_{\dr}(\sigma)  = {\binom{n}{2}}\frac{4n-2}{3} - 8W_{c_3}(\sigma).
\end{equation}
\modifone{Thus, these filtrations are related to each other by a naturally occurring monotone transform that is a function of simplex dimension. However, because of the intrinsic geometric data that gives rise to the filtrations,  the functions $W_{\dr}$ and $W_{c_3}$  are far from inducing similar filtrations, in spite of this close relationship. Specifically, the local directionality filtration of a tournaplex has its vertices in filtration 0, its edges and regular 3-tournaments in filtration 2  and is, roughly speaking,  a refinement of the dimension filtration.  The relationship between the 3-cycle filtration and dimension is more subtle. In filtration 0 one has the sub-tournaplex consisting of all transitive tournaments (in any dimension). In higher stages of the 3-cycle filtration the minimal possible dimension of an added simplex increases with the filtration value (a 3-tournament cannot have a 3-cycle filtration larger than 1), but the dimension of a simplex that is added on in any filtration value is not bounded above. }

In Figure \ref{fig-sphere-abs-dir} we have four digraphs whose flag tournaplexes realise the same homotopy type, but which are distinguished by the associated persistence diagrams with respect to local directionality filtration. Similarly it is easy to see that each of the four digraphs contains a different number of 3-cycles, and hence they are also distinguished by their persistence diagrams corresponding to their 3-cycle filtration. There are however examples of graphs that cannot be distinguished by local directionality or 3-cycle filtration, as in Figure~\ref{fig-sphere-rel-dir}. \modifone{Thus, a related question is whether there could be an advantage in using one of these filtrations as opposed to another, or if there is a point in using a combination of both. In Table \ref{2D-example} we show an example of the two non-isomorphic tournaments in Figure \ref{fig:2tourn}, where the 3-cycle filtration is always contractible (see the bottom row), and where the local directionality filtration is much more interesting (see the right column). This gives a simple answer to the first question. The second will be discussed further below.}

\begin{figure}[ht!]
	\centering
	\includegraphics[width=0.75\textwidth]{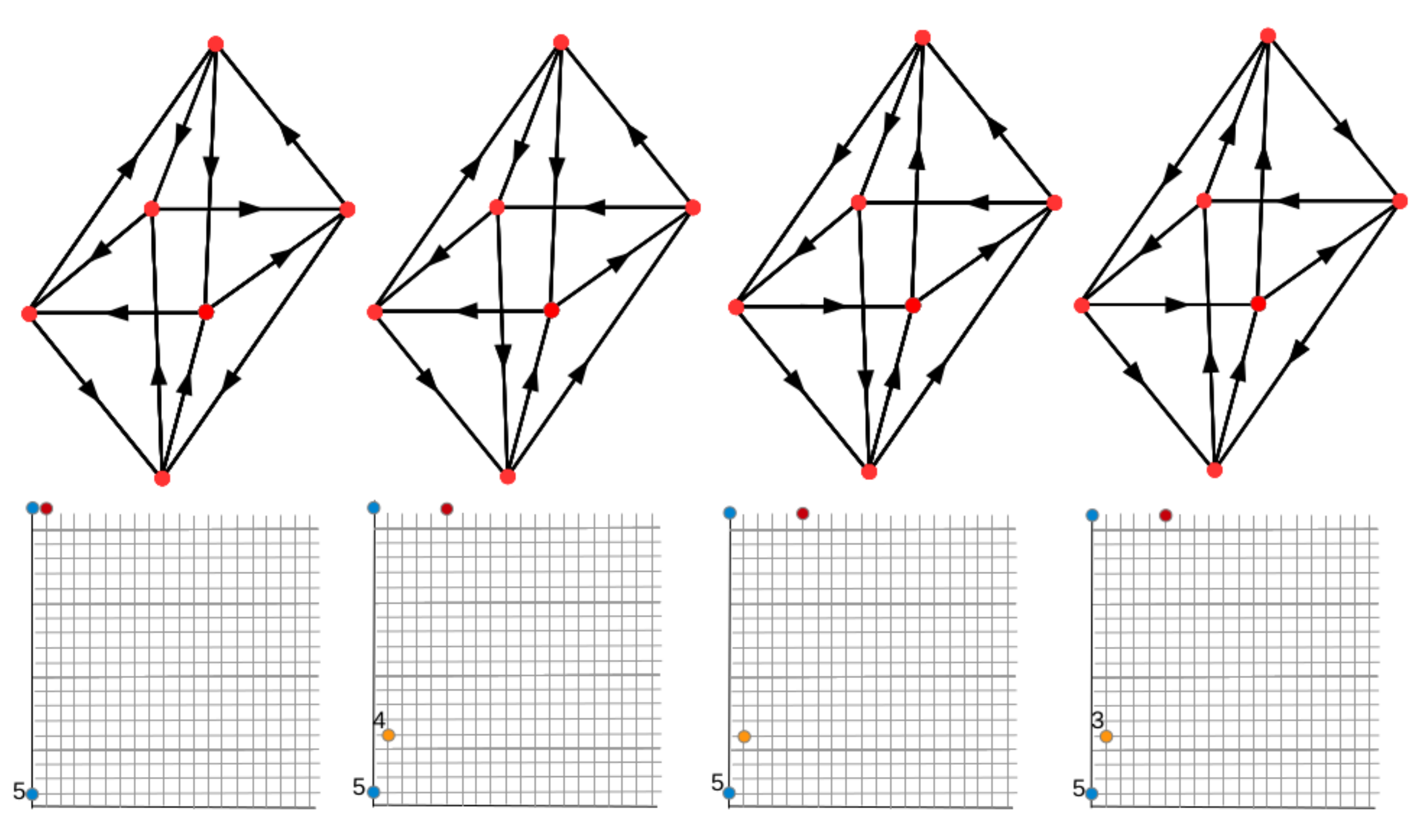}
	\caption{Four non-isomorphic digraphs and their local directionality filtrations reflected in their persistence diagrams. The signed degree of each vertex in each digraph is 0. Hence the global directionality filtration cannot distinguish these graphs. $H_0$, $H_1$ and $H_2$ are denoted in the diagrams by blue, orange and red dots, respectively, and the numbers next to some dots correspond to rank. All four tournaplexes have the homotopy type of a 2-sphere. These graphs can however be distinguished by the homotopy type of their directed flag complexes.}
\label{fig-sphere-abs-dir}
\end{figure}

\begin{Defi}\label{Def-global-dir-filt}
	Let $K$ be a tournaplex. For each tournament $\sigma\in K$, define the \hadgesh{global directionality weight of $\sigma$} to be $W_K(\sigma) = \dr_{K_1}(V_\sigma)=\sum_{v\in V_\sigma} \sdeg_{K_1}(v)^2$, where $K_1$ is the 1-skeleton of $K$ considered as a digraph.
\end{Defi}

\begin{figure}[ht!]
\centering
\includegraphics[width=0.75\textwidth]{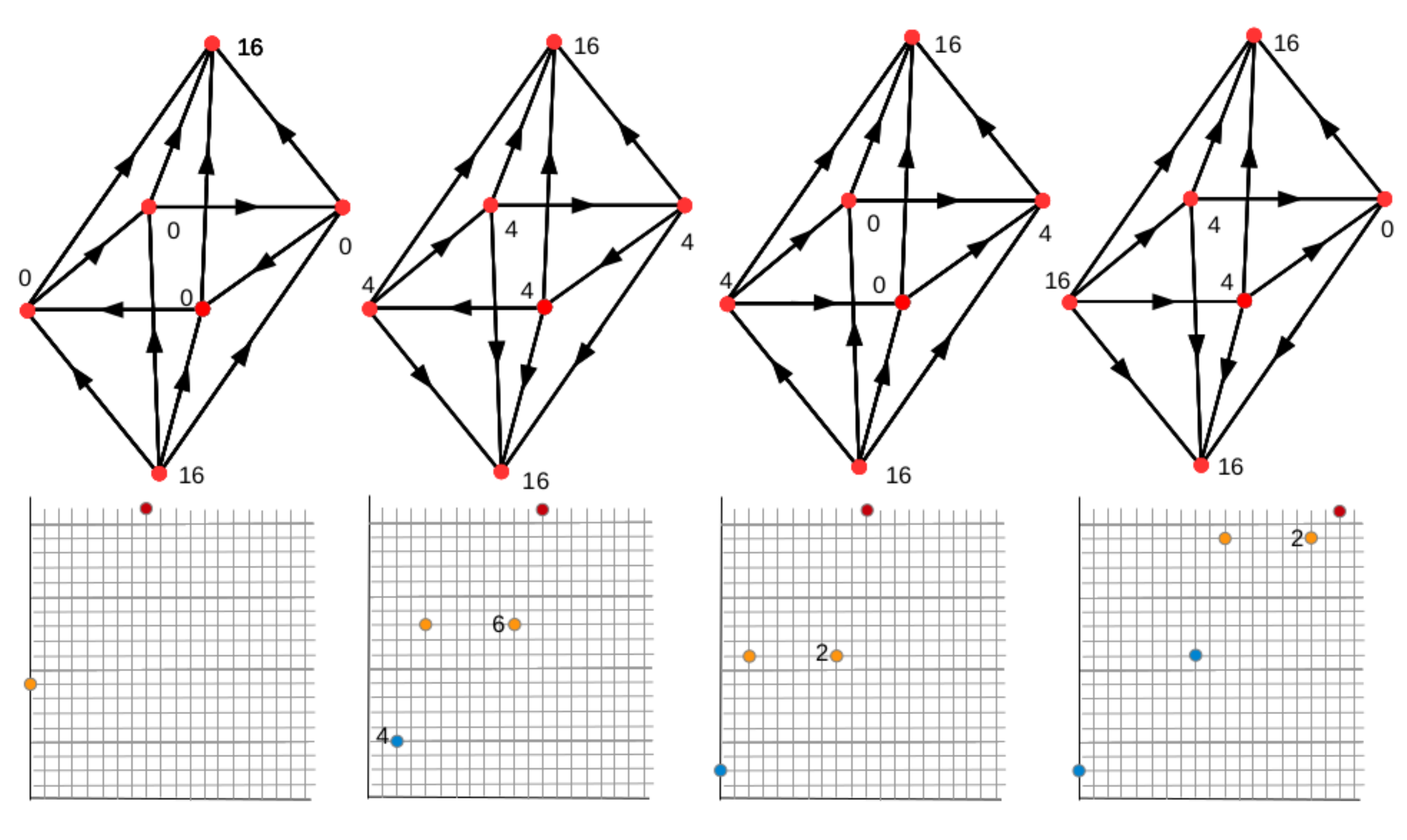}
\caption{Four non-isomorphic digraphs and their global directionality filtrations. The numbers on the vertices denote the corresponding squared signed directionality, and the chart below each digraph is the persistence diagram corresponding to the global directionality filtration on the corresponding tournaplex. $H_0$, $H_1$ and $H_2$ are denoted in the diagrams by blue, orange and red dots, respectively. The numbers next to some dots correspond to rank. Each 3-tournament in all four tournaplexes is transitive. Hence these graphs cannot be distinguished by local directionality filtration or by  3-cycle filtration. }
\label{fig-sphere-rel-dir}
\end{figure}

The global directionality weight function is clearly an increasing filtering function and as such induces an increasing filtration on a tournaplex. Its properties are quite different from the local and the 3-cycle directionality functions in that any positive integer value can be achieved as the global directionality of a tournament in some digraph. See  Figure \ref{fig-sphere-rel-dir} for an example of some digraphs whose flag tournaplexes can be distinguished by the global directionality filtration, but not by the local or the 3-cycle filtrations. On the other hand, the digraphs in Figure \ref{fig-sphere-abs-dir} cannot be distinguished by global directionality.

\begin{figure}[t]\begin{center}\begin{tikzpicture}\def\x{2}\def\y{2}
	\node (0) at (1*\x,0*\y){$0$};
	\node (1) at (0.7*\x,0.7*\y){$1$};
	\node (2) at (0*\x,1*\y){$2$};
	\node (3) at (-0.7*\x,0.7*\y){$3$};
	\node (4) at (-1*\x,0*\y){$4$};
	\node (5) at (-0.7*\x,-0.7*\y){$5$};
	\node (6) at (0*\x,-1*\y){$6$};
	\node (7) at (0.7*\x,-0.7*\y){$7$};
	\foreach \a/\b in {1/0,2/0,2/3,3/0,3/1,3/4,4/0,4/1,5/1,5/2,5/3,6/2,6/3,6/4,6/5,7/0,7/1,7/2,7/3,7/4,7/5,7/6} {
        \draw[black!25,-{Latex[length=3mm, width=2mm]}] (\a) to (\b);
    };
    \foreach \a/\b in {0/5,1/6,2/1,4/2,5/4,6/0} {
        \draw[black,-{Latex[length=3mm, width=2mm]}] (\a) to (\b);
    };
    \end{tikzpicture}
    \hskip 20pt
    \begin{tikzpicture}\def\x{2}\def\y{2}
    \node (0) at (1*\x,0*\y){$0$};
	\node (1) at (0.7*\x,0.7*\y){$1$};
	\node (2) at (0*\x,1*\y){$2$};
	\node (3) at (-0.7*\x,0.7*\y){$3$};
	\node (4) at (-1*\x,0*\y){$4$};
	\node (5) at (-0.7*\x,-0.7*\y){$5$};
	\node (6) at (0*\x,-1*\y){$6$};
	\node (7) at (0.7*\x,-0.7*\y){$7$};
	\foreach \a/\b in {1/0,2/0,2/3,3/0,3/1,3/4,4/0,4/1,5/1,5/2,5/3,6/2,6/3,6/4,6/5,7/0,7/1,7/2,7/3,7/4,7/5,7/6} {
        \draw[black!25,-{Latex[length=3mm, width=2mm]}] (\a) to (\b);
    };
    \foreach \a/\b in {0/5,1/6,2/1,4/2,5/4,6/0} {
        \draw[black,-{Latex[length=3mm, width=2mm]}] (\b) to (\a);
    };
\end{tikzpicture}\end{center}
\caption{The two tournaments $\calg_1$ (left) and $\calg_2$ (right), whose 2D-persistence module is given in Figure~\ref{2D-example}. With their common edges shaded grey, and differing edges in black.}\label{fig:2tourn}
\end{figure}
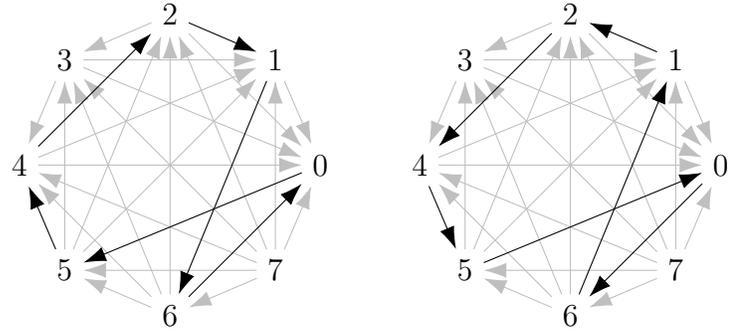

\begin{table}[t]
\begin{tabular}{|c||c|c|c|c|c|c|c|c|}
\hline
\backslashbox{\small{$W_{\dr}$}}{\small{$W_{c_3}$}}&$0$&$1$&$2$&$3$&$4$&$5$&$6$&$9$\\
\hline\hline
$0$&$*_{8}$&$*_{8}$&$*_{8}$&$*_{8}$&$*_{8}$&$*_{8}$&$*_{8}$&$*_{8}$\\
\hline
$2$&$1_{21}$&$1_{12}$&$1_{12}$&$1_{12}$&$1_{12}$&$1_{12}$&$1_{12}$&$1_{12}$\\
\hline
$10$&$2_{26}$&$2_{35}$&$2_{35}$&$2_{35}$&$2_{35}$&$2_{35}$&$2_{35}$&$2_{35}$\\
\hline
$12$&$2_{26}$&$2_{35}$&$2_{22}$&$2_{22}$&$2_{22}$&$2_{22}$&$2_{22}$&$2_{22}$\\
\hline
$20$&$2_{26}$&$2_{16}$&$2_{8}\vee 3_{5}$&$2_{8}\vee 3_{5}$&$2_{8}\vee 3_{5}$&$2_{8}\vee 3_{5}$&$2_{8}\vee 3_{5}$&$2_{8}\vee 3_{5}$\\
\hline
$28$&$3_{12}$&$3_{22}$&$3_{35}$&$3_{35}$&$3_{30}$&$3_{30}$&$3_{30}$&$3_{30}$\\
\hline
$36$&$3_{12}$&$3_{22}$&$3_{35}$&$3_{29}$&$3_{24}$&$3_{24}$&$3_{24}$&$3_{24}$\\
\hline
$44$&$3_{12}$&$3_{22}$&$3_{16}$&\textcolor{red}{$3_{11}\vee 4_1$} \tikz[overlay]{\draw[thick] (-.1,-0.16)--(0.1,0.35);} \textcolor{blue}{$3_{10}$}&$3_{8}\vee 4_{3}$&$3_{8}\vee 4_{3}$&$3_{8}\vee 4_{3}$&$3_{8}\vee 4_{3}$\\
\hline
$52$&$3_{12}$&$3_{8}$&$3_{4}\vee 4_{2}$&$3_{2}\vee 4_{6}$&$3_{2}\vee 4_{11}$&$3_{2}\vee 4_{11}$&$3_{2}\vee 4_{11}$&$3_{2}\vee 4_{11}$\\
\hline
$60$&$*$&$4_{4}$&$4_{10}$&$4_{16}$&$4_{21}$&$4_{21}$&$4_{21}$&$4_{21}$\\
\hline
$62$&$*$&$4_{4}$&$4_{10}$&$4_{16}$&$4_{21}$&$4_{21}$&$4_{18}$&$4_{18}$\\
\hline
$70$&$*$&$4_{4}$&$4_{10}$&$4_{16}$&$4_{21}$&$4_{19}$&$4_{16}$&$4_{16}$\\
\hline
$78$&$*$&$4_{4}$&$4_{10}$&$4_{16}$&$4_{14}$&$4_{12}$&$4_{10}\vee 5_1$&$4_{10}\vee 5_1$\\
\hline
$86$&$*$&$4_{4}$&$4_{10}$&$4_{10}$&$4_{8}$&$4_{6}$&$4_{4}\vee 5_1$&$4_{4}\vee 5_1$\\
\hline
$94$&$*$&$4_{4}$&$4_{4}$&$4_{4}$&$4_{2}$&$4_1\vee 5_1$&$4_1\vee 5_{4}$&$4_1\vee 5_{4}$\\
\hline
$102$&$*$&$*$&$*$&$*$&$5_{2}$&$5_{4}$&$5_{7}$&$5_{7}$\\
\hline
$110$&$*$&$*$&$*$&$*$&$5_{2}$&$5_{4}$&$5_{7}$&$5_{6}$\\
\hline
$134$&$*$&$*$&$*$&$*$&$5_{2}$&$5_{4}$&$5_{4}$&$5_{3}$\\
\hline
$142$&$*$&$*$&$*$&$*$&$5_{2}$&$5_{2}$&$5_{2}$&$5_1$\\
\hline
$150$&$*$&$*$&$*$&$*$&$*$&$*$&$*$&$6_1$\\
\hline
$208$&$*$&$*$&$*$&$*$&$*$&$*$&$*$&$*$\\
\hline
\end{tabular}
\vskip.1in
\caption{The homotopy types of all bifiltration stages of $\tfl(\calg_1)$ and $\tfl(\calg_2)$. Here $*_n$ represents the disjoint union of $n$ points, whereas $m_n$ represents the wedge of $n$ copies of the $m$-sphere. The homotopy types of $\tfl(\calg_1)$ and $\tfl(\calg_2)$ are in red and blue, respectively, and in black when they are the same. }
\label{2D-example}
\end{table}

The multitude of directionality filtrations on tournaplexes suggest that two or more of them can be used in conjunction \modifone{to create a combined filtration or a multi-parameter persistence module. We examined this idea with respect to local directionality and $3$-cycle filtrations. Let $\calg_1$ and~$\calg_2$ be the 8-tournaments depicted in Figure~\ref{fig:2tourn}.
The flag tournaplexes $\tfl(\calg_1)$ and $\tfl(\calg_2)$ have identical 1-parameter  persistence modules with respect to the $3$-cycle directionality filtration and with respect to the local directionality filtration. The 3-cycle filtration stages are all contractible (see the bottom row of Table~\ref{2D-example}). The local directionality filtration is more interesting and is  summarised in (\ref{Eq-homology})} \modifdejan{(cf. also the right column in Table~\ref{2D-example}) below. }

\begin{equation}
H_i(X)=\begin{cases}[0,2)^7\oplus[0,\infty);&i=0,\\
[2,10)^{12};&i=1,\\
[10,12)^{13}\oplus[10,20)^{14}\oplus[10,28)^8;&i=2,\\
[20,28)^5\oplus[28,36)^6\oplus[28,44)^{16}\oplus[28,52)^6\oplus[28,60)^2;&i=3,\\
[44,62)^3\oplus[52,70)^2\oplus[52,78)^6\oplus[60,86)^6\oplus[60,94)^3\oplus[60,102);&i=4,\\
[78,110)\oplus[94,134)^3\oplus[102,142)^2\oplus[102,150);&i=5,\\
[150,208);&i=6,\\
\end{cases}\label{Eq-homology}
\end{equation}
where $X$ is $\tfl(\calg_j)$ with local directionality filtration, for $j=1,2$. 

\modifone{However taking the two filtrations together yields two distinct 2-parameter persistence modules.  Considering the bifiltration of the geometric realisation of each tournaplex, and applying the homotopy theoretic techniques  used in~\cite{Govc}, we were able to compute the homotopy types of each bifiltration pair which are displayed in Figure~\ref{2D-example}. The difference between $\tfl(\calg_1)$ and $\tfl(\calg_2)$ reveals itself in bifiltration $(44,3)$. In both cases we did not compute the full persistence module structure (namely the maps between bifiltration stages).  The point of this calculation was to show that there are examples that are not distinguishable by local directionality or by 3-cycle filtration alone, but a combination of the two does separate them. Similarly one can construct a 1-parameter filtration by taking a function of two variables combining the two filtration functions.  For example, in our case, consider 
\[f(\sigma) \defeq \max\{3W_{\dr}(\sigma), 44 W_{c_3}(\sigma)\}.\]
Then $f(\sigma)\leq 132$ if and only if $W_{\dr}(\sigma)\leq 44$ and $W_{c_3}(\sigma)\leq 3$.
It is then easy to check, given the data in Table \ref{2D-example}, that the filtration function $f$ will distinguish the tournaments in Figure \ref{fig:2tourn} by 1-parameter persistence.}

%%%%%%%%%%%%%%%%%%%%%%%%%%%%%%%%%%%%%%%%%%%%%%%%%%%%%%%%
%%%%% SECTION FOUR         %%%%%%%%%%%%%%%%%%%%%%%%%%%%%%%%%%%%%%%%
%%%%%%%%%%%%%%%%%%%%%%%%%%%%%%%%%%%%%%%%%%%%%%%%%%%%%%%%
%%%%%%%%%%%%%%%%%%%%%%%%%%%%%%%%%%%%%%%%%%%%%%%%%%%%%%%%
%%%%%%%%%%%%%%%%%%%%%%%%%%%%%%%%%%%%%%%%%%%%%%%%%%%%%%%%
%%%%%%%%%%%%%%%%%%%%%%%%%%%%%%%%%%%%%%%%%%%%%%%%%%%%%%%%
%%%%%%%%%%%%%%%%%%%%%%%%%%%%%%%%%%%%%%%%%%%%%%%%%%%%%%%%
%%%%%%%%%%%%%%%%%%%%%%%%%%%%%%%%%%%%%%%%%%%%%%%%%%%%%%%%
%%%%%%%%%%%%%%%%%%%%%%%%%%%%%%%%%%%%%%%%%%%%%%%%%%%%%%%%
%%%%%%%%%%%%%%%%%%%%%%%%%%%%%%%%%%%%%%%%%%%%%%%%%%%%%%%%

\section{Applications}
\label{Sec-applications}
The flag complex associated to a graph reveals higher order properties of the graph as they are encoded in its topology. Similarly, the directed flag complex does the same for digraphs. However, as already pointed out, the directed flag complex ``sees'' tournaments that are not transitive in the graph only by neglecting them, and as such they may or may not be expressed as homology classes in the resulting complex. We introduced the flag tournaplex, motivated exactly by the idea that it will contain more information about the digraph in question than the directed flag complex. This point is clear even without any empirical evidence since in any tournaplex $X$, the subcomplex of transitive tournaments appears as the $0$-th stage in the $3$-cycle filtration. Thus if $X$ is a flag tournaplex, then any higher filtration can only add on the information present in the directed flag complex. In this section we demonstrate by several examples that this is indeed the case.

%\subsection{Local Directionality }
We consider the flag tournaplexes arising from two data sets. The first is a set of directed Erd\H{o}s-R\'enyi type digraphs, and the second is a collection of activity simulation data from the Blue Brain Project that was used in \cite{Fund}. In both cases we consider a set of digraphs divided into subsets by some known parameters, and examine the capability of the topological metrics one can associate to flag tournaplexes with the local directionality filtration to cluster the data into distinct classes, and compare the performance to that of the metrics associated to the directed flag complex. 
 
Given a collection $\mathbb{G}$ of digraphs we wish to partition these graphs into groups of graphs with similar properties.
 In order to do this we must first map each graph ${\calg\in\mathbb{G}}$ to a feature vector $V(\calg)$ of length  $k$ for a suitable positive integer $k$. We shall produce such a vector by two methods. The first is based on a computation of the Betti numbers of directed flag complexes of the graphs in question. The second uses the persistent homology of the flag tournaplex $\tfl(\calg)$ with the local directionality filtration.
 Each method produces a matrix with as many rows as the number of graphs to be clustered and $k$ columns. The rows of these matrices are fed into a $k$-means clustering algorithm to produce our results, as explained below. 
 
\begin{Defi} 
 For a digraph $\calg$, represent the persistent homology of 
 the flag tournaplex $\tfl(\calg)$  with respect to the local directionality filtration as a list of triples $(m,b,d)$, corresponding to  $(dimension, birth, death)$. Let $T^\ell(\calg)$ denote the resulting data set.
 Set $\widehat{T}^\ell(\calg)$ as the set of the unique triples in $T^\ell(\calg)$ and $N_{\calg}(m,b,d)$ as the number of times the triple $(m,b,d)$ appears in $T^\ell(\calg)$.
\end{Defi}

\subsection{Erd\H{o}s-R\'enyi type digraphs with varying local directionality distributions}   
\modifone{In this section we describe a validation experiment we ran where we attempt to show that using local directionality filtration and persistent homology provides greater discriminatory  power than that of Betti numbers of the associated directed flag complex for a fixed artificial dataset.} This is a data set of random graphs denoted $\calg(n,p,q)$ where $n$ is the number of vertices, indexed $1,\ldots, n$, and $p$ and $q$ are probability parameters. An instantiation of $\calg(n,p,q)$ is generated by adding a directed edge $(i,j)$ with probability 
$$\begin{cases}
p,&\mbox{if }i>j\\
q,&\mbox{if }i<j\end{cases}.$$
By varying the probability for different orientations we can control the frequency of tournaments with different local directionality. 

We constructed 50 instantiations in 4 groups with $n=250$, $p=0.25$ and $q = 0, 0.025, 0.05$ and $0.075$, a total of 200 graphs, which we denote by $\mathbb{G}_{ER}$. We use small values for $q$ so that the number of transitive tournaments remains largely unchanged, thus causing minimal difference to the directed flag complexes. To check that this procedure achieved different distributions of local directionality, we computed the average distribution by dimension and directionality. The result is summarised in Figure \ref{fig-ER-Directionality-Dist}.
We then applied Algorithm \ref{alg1} to produce the matrix $V^{6}_\calt(\mathbb{G}_{ER})$.

\begin{Rmk}
\modif{The aim of Algorithm 1 is to use persistent homology of tournaplexes to extract the key functional information from the spike trains and express said data as a feature vector that can be used for classification. We begin by applying persistent homology and expressing the resulting information as a vector, this is only possible since we know the directionality filtration has a finite number of possible values. However, the dimension of the resulting vector is too high to apply clustering algorithms. So the remaining steps of Algorithm 1 are a form of feature selection designed to extract the most important entries of the vector. This is done by using standard deviation and selecting the filtration value and time bin pairs which vary the most, the intuition being that these entries are the most likely to vary between input, thus giving the best classification accuracy.}
\end{Rmk}

\begin{figure}[ht!]
\centering
\includegraphics[width=1\textwidth]{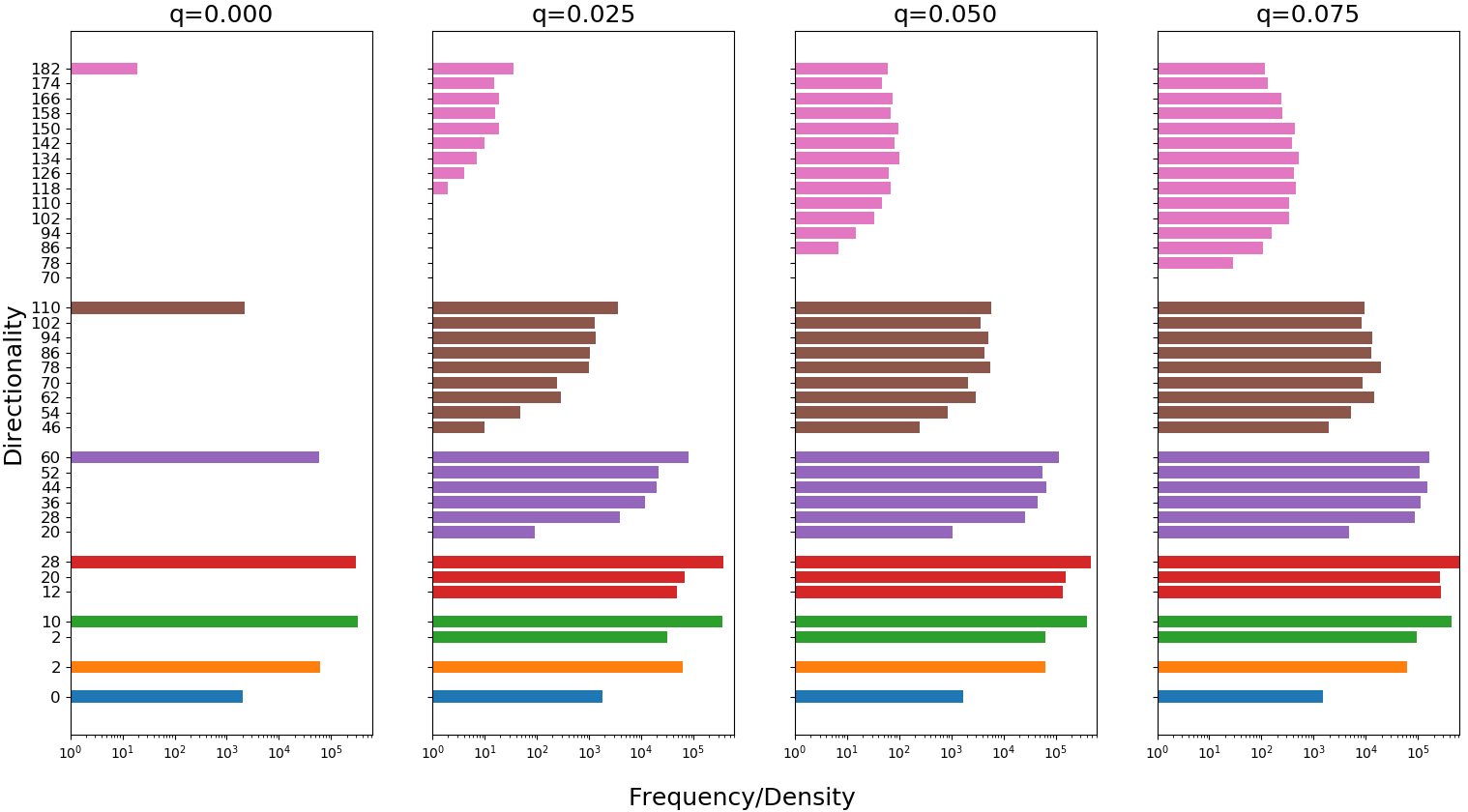}
\caption{\modif{Distribution of local directionalities as a function of the probability parameter $q$, normalised by connection density. The $x$-axis is labeled by Frequency/Density.} The different colours correspond to \modif{the number $n$ of vertices in a tournament, with blue corresponding to a single vertex and pink to 7 vertices. The vertical axis is labeled by the possible local directionality weights in each dimension.}}
\label{fig-ER-Directionality-Dist}
\end{figure}

\begin{algorithm}[ht!]
\caption{Computing local directionality feature matrix $V^d_\mathcal{T}(\mathbb{G})$}\label{alg1}
\renewcommand{\algorithmicrequire}{\textbf{Input:}}
\renewcommand{\algorithmicensure}{\textbf{Output:}}
\renewcommand{\algorithmicprocedure}{\textbf{Procedure:}}
\begin{algorithmic}[1]
\Require A set $\mathbb{G}=\{\calg_1,\ldots,\calg_n\}$ of digraphs, and an integer $d>0$
\Procedure{}{}
\State Compute persistent homology $T^\ell(\calg_i)$ of $\tfl(\calg_i)$ with respect to $W_{\dr}$ for all $\calg_i\in\mathbb{G}$
\State Set $\widehat{T}=\bigcup_{i=1}^n \widehat{T}^\ell(\calg_i)$, and fix an arbitrary ordering $t_1,\ldots,t_r$ on its elements
\State Set $M = (m_{i,j})$ as the $n\times r$ matrix with $m_{i,j}=N_{\calg_i}(t_j)$
\State Compute the standard deviation of each column of $M$ (considered as a set of integers)
\State Let $C_1,\ldots,C_d$ be the $d$ columns with the largest standard deviation
\State Return the $n\times d$ matrix $V^d_\mathcal{T}(\mathbb{G}) $, given by concatenating the columns $C_1\cdots C_d$
\EndProcedure
\end{algorithmic}
\end{algorithm}

Let $V_\beta^6(\mathbb{G}_{ER})$ denote the matrix where the $i$-th row is the vector $V_\beta^6(\calg_i)\defeq[\beta_0,\beta_1,\ldots,\beta_5]$ of Betti numbers
 of the directed flag complex of $\calg_i\in\mathbb{G}_{ER}$.
 Next, we applied $k$-means clustering to the rows of each of the matrices $V_\calt^6(\mathbb{G}_{ER})$ and $V_\beta^6(\mathbb{G}_{ER})$,
 using Python package scikit-learn \cite{scikit}.
 The results are displayed in Figure~\ref{fig:kmeans1}. Comparing the results to the distribution of directionalities in Figure \ref{fig-ER-Directionality-Dist}, one notices that the distribution of transitive tournaments remains roughly the same regardless of $q$. Hence, it is to be expected that the Betti numbers of the directed flag complex will perform poorly in clustering the four families, which is indeed the case. By contrast, the associated tournaplexes, filtered by local directionality give almost perfect separation, which is particularly remarkable in the cases $q=0.05$ and~$q=0.075$  that give very similar distributions in Figure \ref{fig-ER-Directionality-Dist}. This demonstrates that the topology of the tournaplex holds more information about  the orientation of the edges in a digraph, compared to the directed flag complex.

\begin{Rmk}\label{Rem-6}
We select $d=6$ in this computation for two reasons. Firstly, the Betti numbers are always zero for dimension $6$ and above in the directed flag complex. Secondly, for  $k$-means clustering to give reliable results we require that the size of the vectors to be clustered is significantly  smaller than the size of the data set. 
Also, we use $k$-means clustering as it is a simple well known technique, but  similar results are obtained using decision tree learning or linear discriminant analysis, \modif{as verified by Henri Riihim\"aki.}
\end{Rmk}

\begin{figure}[ht!]
\centering
\begin{tikzpicture}[scale=0.3]
\foreach \x in {1,...,50}{
  \ifthenelse{\x=37 \OR \x=38}
  	{\node[diamond,inner sep=0pt,text width=2mm,fill=Blue] at (\x,0){};}
  	{\node[regular polygon,regular polygon sides=3,inner sep=0pt,text width=.8mm,fill=Red] at (\x,0){};}
  \node[diamond,inner sep=0pt,text width=2mm,fill=Blue] at (\x,1){};
  \node[circle,inner sep=0pt,text width=1.5mm,fill=Green] at (\x,2){};
  \node[regular polygon,regular polygon sides=4,inner sep=0pt,text width=.8mm,fill=YellowOrange] at (\x,3){};
}
\node at (25,4) {Tournaplex};
\node at (-2,3) {\tiny $0.000$};
\node at (-2,2) {\tiny $0.025$};
\node at (-2,1) {\tiny $0.050$};
\node at (-2,0) {\tiny $0.075$};
\node at (-4,1.5) {\tiny $q$};
\end{tikzpicture}
\begin{tikzpicture}[scale=0.3]
\foreach \x in {1,...,50}{
  \ifthenelse{\x>14}
  	{\node[regular polygon,regular polygon sides=4,inner sep=0pt,text width=.8mm,fill=YellowOrange] at (\x,0){};}
  	{\ifthenelse{\x=1 \OR \x=2 \OR \x=3 \OR \x=4 \OR \x=9 \OR \x=14}
  		{\node[regular polygon,regular polygon sides=3,inner sep=0pt,text width=.8mm,fill=Red] at (\x,0){};}
  		{\node[diamond,inner sep=0pt,text width=2mm,fill=Blue] at (\x,0){};}}
  \node[regular polygon,regular polygon sides=4,inner sep=0pt,text width=.8mm,fill=YellowOrange] at (\x,1){};
  \ifthenelse{\x>11}
  	{\node[regular polygon,regular polygon sides=4,inner sep=0pt,text width=.8mm,fill=YellowOrange] at (\x,2){};}
  	{\node[circle,inner sep=0pt,text width=1.5mm,fill=Green] at (\x,2){};}
  \node[regular polygon,regular polygon sides=4,inner sep=0pt,text width=.8mm,fill=YellowOrange] at (\x,3){};
}
\node at (25,4) {Directed Flag Complex};
\node at (-2,3) {\tiny $0.000$};
\node at (-2,2) {\tiny $0.025$};
\node at (-2,1) {\tiny $0.050$};
\node at (-2,0) {\tiny $0.075$};
\node at (-4,1.5) {\tiny $q$};
\end{tikzpicture}
\caption{The cluster assignment of $k$-means clustering applied to the columns of the matrices $V^6_\mathcal{T}(\mathbb{G})$ (top)
 and $V^6_\beta(\mathbb{G})$ (bottom),
 where $\mathbb{G}$ is a set of 200 directed Erd\H{o}s-R\'enyi type graphs with parameters $p=0.25$ and varying~$q$. The four colours represent the different clusters, each row corresponds to a different value of $q$.}\label{fig:kmeans1}
\end{figure}
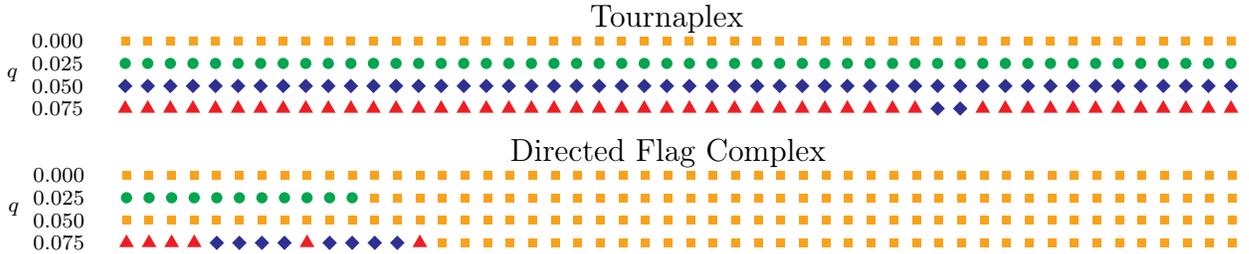

\subsection{Brain activity simulation digraphs}\label{Sec4.2} 
The data in this example was taken from Blue Brain Project's reconstruction of the neocortical column of a juvenile rat \cite{Mega}. The reconstruction is a digital model of brain tissue based on basic biological principles. The particular model  we used contains approximately 31000 digitally simulated neurons with roughly 8 million synaptic connections. There are 55 morphological types of neurons simulated in the column, in correspondence with known biological classification, and the column is built in 6 layers, again following known biological principles. With the possible exception of more recent developments by the Blue Brain Project team, this model is  the most accurate approximation of real brain tissue currently available. In particular it allows researchers to obtain full connectivity data of the neural microcircuitry, and can be stimulated by spike trains that are either recorded in vivo or artificially generated.

The  data set we use here was used in \cite{Fund} to demonstrate the capability of the homology of the directed flag complex to express certain properties of the simulations.  In this data set nine different stimuli were applied to the Blue Brain Project microcircuit. These stimuli, which take the form of spike trains that are injected directly into the digital tissue, can be distinguished by two properties, their \emph{class} and their \emph{grouping}, with three possible values for each property.
 Each class, denoted $5$, $15$ or $30$, represents a different temporal input of the stimulus,
 and each grouping, denoted $a$, $b$ or $c$, represents  a different spatial input of the stimulus.
 See \cite[Figure 4A]{Fund} for further information about the stimuli. Rather than using the entire microcircuit, we extracted spike trains only from L5TTPC1 neurons \cite{Mega}. The code name stands for \hadgesh{Layer 5 Thick Tufted Pyramidal Cell of Type 1}, and these neurons are one of the most prevalent morphological types in the neocortex and are assumed to be of major significance in information processing in the brain. The structural data is publicly available at \cite{rat}. 

 \begin{algorithm}[t]
\caption{Computing local directionality feature matrix $\widehat{V}^m_\mathcal{T}(\mathbb{D})$ for functional data}\label{alg2}
\renewcommand{\algorithmicrequire}{\textbf{Input:}}
\renewcommand{\algorithmicensure}{\textbf{Output:}}
\renewcommand{\algorithmicprocedure}{\textbf{Procedure:}}
\begin{algorithmic}[1]
\Require Spike trains $\mathbb{D}=\{D_1,\ldots,D_n\}$ on a graph $\calg$ and three positive integers $\{m, t_1, t_2\}$
\Procedure{}{}
\For{$i\in\{1,\ldots,n\}$}
\For{$j\in\{1,\ldots,\lceil L/t_1\rceil\}$}
\State Compute transmission-response graphs $\calg_j^i\defeq\calg_j^{D_i}$ with parameters $\{t_1, t_2\}$
\State  Compute  persistent homology $T^\ell(\calg_j^i)$ of~$\tfl(\calg_j^i)$ with respect to $W_{\dr}$
\EndFor
\EndFor
\State Set $\widehat{T}=\cup_{i,j} \widehat{T}^\ell(\calg_j^i)$ and fix an arbitrary ordering $t_1,\ldots, t_l$ on its elements
\State Set $M_j = (m_{j,c}^{i})$ to be the $n\times l$ matrix with $m_{j,c}^{i}=N_{\calg_j^i}(t_c)$
\State Let $M$ be the concatenation of the matrices $M_1, M_2,\ldots, M_{\lceil L/t_1\rceil}$ 
\State Compute the standard deviation of each column of $M$ (considered as a set of integers)
\State Let $C_1,\ldots,C_m$ be the $m$ columns with the largest standard deviation
\State Return the $n\times m$ matrix $\widehat{V}^m_\mathcal{T}(\mathbb{D})$ given by  concatenating the columns $C_1\cdots C_m$
\EndProcedure
\end{algorithmic}
\end{algorithm}

We consider five repetitions of each experiment, referred to as \emph{seed} $s$, for $s=1,\ldots,5$.
 Thus we have $45$ distinct data sets, each consisting of a 250ms spike train, that is a list of pairs
 $(t,g)$ where $t$ is a time (in resolution of 0.1ms) and $g$ is a neuron (vertex) that spiked at  time $t$. 
Each data set is then converted to a sequence of digraphs, using the  \emph{transmission-response} method introduced in \cite{Fund}. The construction relies on three items of input: 1) the underlying structural graph $\calg$, 2) a spike train data set $D$ and 3)  a pair of integers $t_1$ and~$t_2$.  We then split the spike train into time bins of size $t_1$-ms, and construct a graph~$\calg_{r}$ for each time bin, where $\calg_{r}$ has the same vertices of $\calg$
 and the edge $(i,j)$ if all of the following conditions hold:
 \begin{enumerate}[(a)]
 \item $(i,j)$  is an edge in $\calg$,
 \item neuron $i$  fired at time $t_0$ in the $r$-th time bin, and
 \item neuron $j$  fired at time $t_0<t\le t_0+t_2$.
 \end{enumerate}
 See \cite{Fund} for further background on transmission-response graphs. 

\begin{algorithm}[t]
\caption{Computing Betti number feature vector $\widehat{V}^m_\beta(\mathbb{D})$ for functional data}\label{alg3}
\renewcommand{\algorithmicrequire}{\textbf{Input:}}
\renewcommand{\algorithmicensure}{\textbf{Output:}}
\renewcommand{\algorithmicprocedure}{\textbf{Procedure:}}
\begin{algorithmic}[1]
\Require Spike trains $\mathbb{D}=\{D_1,\ldots,D_n\}$ on a graph $\calg$ and four positive integers $\{d, m, t_1, t_2\}$
\Procedure{}{}
\For{$i\in\{1,\ldots,n\}$}
\For{$j\in\{1,\ldots, \lceil L/t_1\rceil\}$}
\State Compute transmission-response graphs $\calg_j^i\defeq\calg_j^{D_i}$, with parameters $\{t_1,t_2\}$
\EndFor
\EndFor
\State Set $M_j$ as the $n\times m$ matrix with the $i$-th row as $V_\beta^d(\calg_j^i)$
\State Let $M$ be the concatenation of the matrices $M_1, M_2,\ldots, M_{\lceil L/t_1\rceil}$
\State Compute the standard deviation of each column of $M$
\State Let $C_1,\ldots,C_m$ be the $m$ columns with the largest standard deviation
\State Return the $n\times m$ matrix $\widehat{V}^m_\beta(\mathbb{D})$ given by concatenating the columns $C_1\cdots C_m$
\EndProcedure
\end{algorithmic}
\end{algorithm}
 
A typical collection $\mathbb{D}$ of spike train data sets consists of $n$ distinct instantiations, denoted by $D_r$, $r=1,\ldots, n$, each of length L ms. Out of this collection we produce an $n\times m$ feature matrix $\widehat{V}^m_\mathcal{T}(\mathbb{D})$, for some natural number $m$. The rows of this matrix are then fed  into a $k$-means clustering algorithm. To obtain  $\widehat{V}^m_\mathcal{T}(\mathbb{D})$ we apply Algorithm~\ref{alg2}, which is a similar procedure to Algorithm~\ref{alg1}, but is adapted for use with functional data. As noted above, for the collection $\mathbb{D}$ at our disposal we have $n=45$ and $L=250$, and we set $m=6$ so that the size of our vectors is suitable compared to the size of the data set (see Remark \ref{Rem-6}). 

For the directed flag complex we compute $\widehat{V}^m_\beta(\mathbb{D})$ by applying Algorithm~\ref{alg3}, which is very similar to Algorithm~\ref{alg2}, but using the Betti numbers instead of persistent pairs. We only use the first three Betti numbers ($d=3$ in the algorithm), as in this data set all higher Betti numbers were zero.

Once we have computed both $\widehat{V}^6_\beta(\mathbb{D})$ and $\widehat{V}^6_\mathcal{T}(\mathbb{D})$,  we apply $k$-means clustering to their rows.
 The results of this are displayed in Figure~\ref{fig:kmeans2}.
  Here too one sees that the data is split  into three classes almost perfectly using $k$-means clustering on the corresponding flag tournaplexes filtered by local directionality, while applying the analogous clustering method on the  Betti numbers of the  directed flag complexes yields poorer separation. This suggests once more that the tournaplex construction stores different, and possibly more detailed information about the network than the directed flag complex.

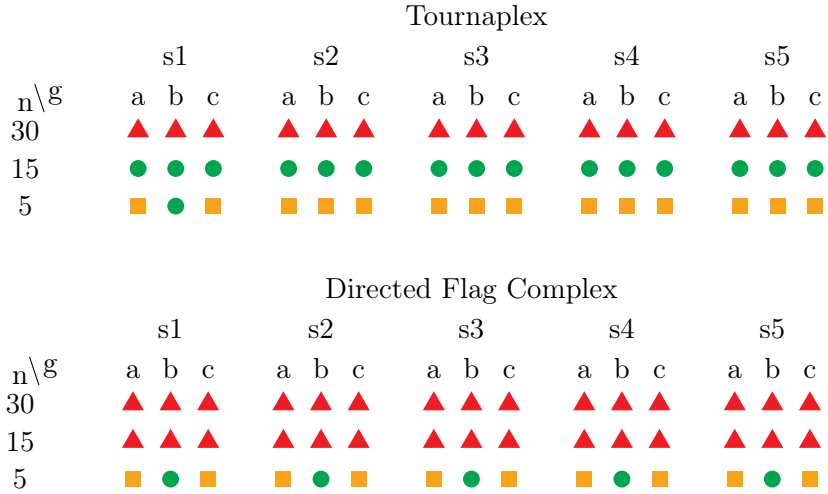
\begin{figure}[t]
\centering
\begin{tikzpicture}[scale=0.5]
\foreach \x in {1,...,19}{
  \ifthenelse{\x=4 \OR \x=8 \OR \x=12 \OR \x=16}{}{
  \ifthenelse{\x=2}
  	{\node[circle,inner sep=0pt,text width=2.25mm,fill=Green] at (\x,0){};}	
  	{\node[regular polygon,regular polygon sides=4,inner sep=0pt,text width=1.5mm,fill=YellowOrange] at (\x,0){};}
  \node[circle,inner sep=0pt,text width=2.25mm,fill=Green] at (\x,1){};
  \node[regular polygon,regular polygon sides=3,inner sep=0pt,text width=1.2mm,fill=Red] at (\x,2){};
}
}
\foreach \x in {1,5,9,13,17}{\node[above] at (\x,2.4) {\small a};}
\foreach \x in {2,6,10,14,18}{\node[above] at (\x,2.4) {\small b};}
\foreach \x in {3,7,11,15,19}{\node[above] at (\x,2.4) {\small c};}
\node at (-2,0) {\small 5};
\node at (-2,1) {\small 15};
\node at (-2,2) {\small 30};
\node at (2,4) {\small s1};\node at (6,4) {\small s2};\node at (10,4) {\small s3};\node at (14,4) {\small s4};\node at (18,4) {\small s5};
\draw[black] (-1.5,2.5) to (-1.75,3.25);\node at (-2,2.7) {\small n};\node[above] at (-1.2,2.4) {\small g};
\node at (10,5) {\small Tournaplex};
\end{tikzpicture}

\vskip 15pt

\begin{tikzpicture}[scale=0.5]
\foreach \x in {1,...,19}{
  \ifthenelse{\x=4 \OR \x=8 \OR \x=12 \OR \x=16}{}{
  	\ifthenelse{\x=2 \OR \x=6 \OR \x=10 \OR \x=14 \OR \x=18}
  		{\node[circle,inner sep=0pt,text width=2.25mm,fill=Green] at (\x,0){};}
  		{\node[regular polygon,regular polygon sides=4,inner sep=0pt,text width=1.5mm,fill=YellowOrange] at (\x,0){};}
  \node[regular polygon,regular polygon sides=3,inner sep=0pt,text width=1.2mm,fill=Red] at (\x,1){};
  \node[regular polygon,regular polygon sides=3,inner sep=0pt,text width=1.2mm,fill=Red] at (\x,2){};
}
}
\foreach \x in {1,5,9,13,17}{\node[above] at (\x,2.4) {\small a};}
\foreach \x in {2,6,10,14,18}{\node[above] at (\x,2.4) {\small b};}
\foreach \x in {3,7,11,15,19}{\node[above] at (\x,2.4) {\small c};}
\node at (-2,0) {\small 5};
\node at (-2,1) {\small 15};
\node at (-2,2) {\small 30};
\node at (2,4) {\small s1};\node at (6,4) {\small s2};\node at (10,4) {\small s3};\node at (14,4) {\small s4};\node at (18,4) {\small s5};
\draw[black] (-1.5,2.5) to (-1.75,3.25);\node at (-2,2.7) {\small n};\node[above] at (-1.2,2.4) {\small g};
\node at (10,5) {\small Directed Flag Complex};
\end{tikzpicture}
\caption{The cluster assignment of $k$-means clustering applied to $\widehat{V}^6_\mathcal{T}(\mathbb{D})$ (top) and $\widehat{V}^6_\beta(\mathbb{D})$ (bottom),
 where $\mathbb{D}$ is a set of spike trains on the Blue Brain microcircuit reconstruction.}\label{fig:kmeans2}
\end{figure}
\bigskip

On behalf of all authors, the corresponding author states that there is no conflict of interest.

%\bibliographystyle{abbrv}
%\bibliography{bibfile}

 \end{document}